\documentclass[12pt]{article}
\usepackage[utf8]{inputenc}
\usepackage{amsmath}
\usepackage{amsfonts}
\usepackage{amssymb}
\usepackage{amsthm}
\usepackage{bbold}
\usepackage{mathrsfs}
\usepackage{IEEEtrantools}
 \usepackage{graphicx}
 \usepackage{tikz-cd}
 \usepackage[shortlabels]{enumitem}

\newcommand{\Nn}{\mathbb{N}}

\newcommand{\Pp}{\mathbb{P}}
\newcommand{\Qq}{\mathbb{Q}}

\newcommand{\tp}{\mathrm{tp}}

\newcommand{\acl}{\mathrm{acl}}

\newcommand{\dcl}{\mathrm{dcl}}
\newcommand{\e}{\epsilon}
\newcommand{\si}{\sigma}

\renewcommand{\S}{\mathcal{S}}

\newcommand{\I}{\mathbb{1}}
\newcommand{\Om}{\Omega}

\def\Ind#1#2{#1\setbox0=\hbox{$#1x$}\kern\wd0\hbox to 0pt{\hss$#1\mid$\hss}
\lower.9\ht0\hbox to 0pt{\hss$#1\smile$\hss}\kern\wd0}

\def\ind{\mathop{\mathpalette\Ind{}}}

\def\notind#1#2{#1\setbox0=\hbox{$#1x$}\kern\wd0
\hbox to 0pt{\mathchardef\nn=12854\hss$#1\nn$\kern1.4\wd0\hss}
\hbox to 0pt{\hss$#1\mid$\hss}\lower.9\ht0 \hbox to 0pt{\hss$#1\smile$\hss}\kern\wd0}

\begin{document}

\newtheorem{lemma}{Lemma}[subsection]
\newtheorem{theorem}[lemma]{Theorem}
\newtheorem{definition}[lemma]{Definition}
\newtheorem{example}[lemma]{Example}
\newtheorem{corollary}[lemma]{Corollary}
\newtheorem{proposition}[lemma]{Proposition}
\newtheorem{claim}{Claim}[lemma]

\title{An Algebraic Hypergraph Regularity Lemma} 
 
\author{Alexis Chevalier, Elad Levi}

\date{}

\maketitle 
 
\begin{abstract}
Szemerédi's regularity lemma is a powerful tool in graph theory. It states that for every large
enough graph, there exists a   partition of the edge set with bounded size such that most induced subgraphs are quasirandom.
 
  When the graph is a
definable set $\phi(x, y)$ in a finite field $F_q$, Tao's algebraic graph regularity lemma (\cite{Tao2012}) shows that there is a partition of the graph $\phi(x, y)$ such that all induced subgraphs are quasirandom and the error bound on quasirandomness is $O(q^{-1/4})$. 

In this work we prove an algebraic hypergraph regularity lemma for definable sets in finite fields, thus answering a question of Tao. We also extend the algebraic regularity lemma to definable sets in the difference fields $(F_q^{alg}, x^q)$ and we offer a new point of view on the geometric content of the algebraic regularity lemma.
\end{abstract}
 
 \tableofcontents
  
 \section{Introduction}
 In this paper, we prove a strong hypergraph regularity lemma for definable sets in finite fields, answering a question of Tao. We offer three main contributions which build on the algebraic regularity lemma of \cite{Tao2012}. Firstly, we extend the algebraic   regularity lemma  to arbitrary definable uniform hypergraphs, of any arity. Secondly, we extend the algebraic regularity lemma beyond the context of finite fields by  proving our results  in the difference fields $K_q  = (F_q^{alg}, x^q)$ for definable sets in the language of rings with a difference operator $\si$.
   Thirdly, we offer a new point of view on algebraic  regularity, relating combinatorial Szemer\'edi-style regularity for definable sets to algebraic and geometric  properties of associated varieties and function fields. Therefore our hypergraph partitions have a natural geometric interpretation. 

The algebraic regularity lemma of \cite{Tao2012} says that a strong version of Szemer\'edi regularity holds for definable sets $\phi(x, y)$  (with parameters) when interpreted in finite fields $F_q$.   Suppose that $\phi(x, y)$ is contained in the definable set $X \times Y$.  The algebraic regularity lemma says that there is some $N \geq 1$ and definable sets $X_1, \ldots, X_{n}$ and $Y_1, \ldots, Y_{m}$  partitioning $X$ and $Y$ respectively  with $n, m\leq N$ such that for all $i\leq n$ and $j \leq m$, the graph $\phi(x,y) \wedge X_i(x) \wedge Y_j(y)$ is $q^{-1/4}$-regular in the field $F_q$. This means that for sufficiently large prime powers $q$, taking parameters in $F_q$, if  $\phi \cap X_i \times Y_j \neq \emptyset$ in $F_q$, then for any $A \subseteq X_i$ and $B \subseteq Y_j$, 
\[
|\frac{|\phi(x, y) \cap A \times B|}{|\phi(x, y) \cap  X_i \times Y_j|} - \frac{|A \times B|}{|X_i \times Y_j|}|  = O(q^{-1/4}).
\]
$O(\cdot)$ and $N$ depend only on the formula $\phi$ and not on the parameters.

The algebraic regularity lemma  for graphs strengthens the classical Szemer\'edi regularity lemma   in two ways. Firstly, there is a fixed $N$ such that the error bounds on regularity vanish against a fixed partition of the graph into at most $N^2$ subgraphs, as $q \to \infty$. As a result, the sizes of the sets $X_i$ and $Y_j$ are of the same order as $|X|$ and $|Y|$ as $q \to \infty$. Secondly, regularity is obtained for all pairs $(X_i, Y_j)$, whereas the classical graph regularity lemma only guarantees regularity for \emph{most} pairs, in an appropriate sense.  The Szemer\'edi regularity lemma first appeared   in \cite{Szemeredi1976} and the reader is referred to \cite{Gowers2006} for a general discussion.

The algebraic regularity lemma of \cite{Tao2012} has   attracted considerable attention from model theorists and it is now well-understood that the main engine behind this result is the stability of the formulas  $\mu(\phi(x, a) \wedge \psi(x, b)) = \lambda$, where $\mu$ is the counting measure in the variable $x$ which was first shown to be definable in  \cite{ChatzidakisVanDDMacintyre1992}.  See \cite{Pillay2013} for a model theoretic discussion of the algebraic regularity lemma for graphs and its extension to a wider class of  theories. 

However, generalising Tao's algebraic regularity lemma to definable hypergraphs requires new ideas.  In this paper, we take a fundamentally geometric approach to this problem and we use the model theory of ACFA  to derive an algebraic hypergraph regularity lemma which strengthens the classical hypergraphs regularity lemmas of 
  \cite{Gowers2007} and \cite{RodelSkokan2004}  in the same ways as Tao's algebraic regularity lemma strengthens Szemer\'edi's. 
  
  There are some interesting and surprising technical twists to our results, but given a definable hypergraph $\phi(x_1, \ldots, x_n) \subseteq X_1 \times \ldots \times X_n$ in finite fields (or in the difference fields $K_q$, as discussed below), we find some $N \geq 0$ not depending only on $\phi$ and a partition $\mathcal{W}_u$ of each set $\prod_{i \in u} X_i$ where $u \subseteq \{1, \ldots, n\}$ and $|u| = n-1$ such that $|\mathcal{W}_u| \leq N$ for every $u$ and such that restricting $\phi$ to any family $(W_u)_{|u| = n-1}$ with $W_u \in \mathcal{W}_u$ gives a $q^{-2^{-n-1}}$-regular hypergraph in the sense of \cite{Gowers2007}. See Theorem \ref{theorem: combinatorial regularity, classical partitions} for a precise statement. 
%  
%We remark here that some strong results about hypergraph regularity have already been achieved thanks to model theoretic tools. \cite{ChernikovStarchenko2016} shows that the classical hypergraph regularity lemmas can be strengthened for definable hypergraphs in NIP structures. See also \cite{ChernikovTowsner2020} for an extension of those results to functions.
% Our results lie at the opposite end of the model theoretic spectrum, since we work here in a  simple unstable theory. 
% 
% The recent results of \cite{TerryWolf2021} apply in a setting which is closer to ours but rely on a fundamentally different combinatorial approach and produce results of a different nature. In that work, the authors seek to generalise the stability-theoretic proof of Tao's result to hypergraphs. Our approach is much more geometric and is  very much tied to the context of pseudofinite fields or ACFA. It is not yet clear  if and how our approach can be adapted to other settings.
 
 \bigskip

The choice to work with difference equations and ACFA rather than pseudofinite fields is a natural and useful one. Recall that ACFA is the model completion of the theory of algebraically closed inversive difference fields in the language of rings with a difference operator $\si$. ACFA extends naturally the theory of pseudofinite fields in the sense that for any $K \models ACFA$, the fixed field $\si(x) = x$ of $K$ is a model of the theory of pseudofinite fields and is stably embedded inside $K$.  The fundamental results of \cite{HrushovskiFrob} also show that ACFA is the asymptotic theory of the structures $K_q = (F_q^{alg}, x^q)$ in the same way that the theory of pseudofinite fields is the asymptotic theory of the finite fields $F_q$.

In ACFA, there is a form of quantifier-elimination which is more natural than the quantifier-elimination of pseudofinite fields: every definable set contained in a variety of finite total dimension is equivalent to the image of a projection $\pi : X \to Y$  where $X$ and $Y$ are difference varieties of finite total dimension and $X$ is a finite cover of $Y$.
 In this paper, we refer to such definable sets as \emph{Galois formulas.} Galois formulas specialise to finite sets in the structures $K_q$ and it is easy to deduce a natural characterisation of the counting measure on Galois formulas using the results of \cite{HrushovskiFrob}. This unlocks an algebraic characterisation of combinatorial regularity. 

We note that it is possible to recover the  same kind of quantifier-elimination as in ACFA in the more modest context of  pseudofinite fields by enriching the language of rings slightly. If  $F$ is a pseudofinite field viewed as a structure in the language of rings, we can add a sort $F_n$ for every   Galois extension of $F$ of degree $n$ and on each $F_n$ we add a difference operator $\si_n$ which we interpret as a generator of $Gal(F_n/F)$.  See \cite{Johnson2019} for an example of this approach. Under this approach and in this restricted context, all results in this paper could be recast using   the classical Lang-Weil estimates and the discussion of the counting measure found in the Appendix of \cite{Hrushovski2002}.

However, instead of working with the expansion of  pseudofinite fields described above, it is more natural to work with full models of ACFA. Hrushovski's twisted Lang-Weil estimates directly extend the classical Lang-Weil estimates to this setting, so the technical transition is seamless. Moreover, shifting the focus to difference equations  gives access to many new definable sets which do not come from pseudofinite fields. For example, algebraic dynamics, which is concerned with the study of equations of the form $\si(x) = f(x)$ where $f$ is rational, produces many new definable sets of great interest. Finite simple groups of Lie type also fall under this umbrella. 

In fact,  much of the  proof of our algebraic hypergraph regularity lemma happens in a setting which is similar to the category of algebraic dynamics defined in \cite{ChatzidakisHrushovski2008}.  If $\phi$ is a Galois formula corresponding to the finite projection of irreducible difference varieties $X \to Y$  over a difference field $A$, then we associate to $\phi$ a finite Galois extension $L/A(a)_\si$ where $A(a)_\si$ is the difference function field of $Y$ and $L$ is the pure field associated to  the finite cover $Y$ over $A(a)_\si$.  Our fundamental object of study is  the pair $(A(a)_\si, L)$, where $L$   has no specified difference structure. We think of $\phi$ as asking a question about extensions of $\si$ from $A(a)_\si$ to $L$. 
In this setting, we find that  combinatorial regularity for   definable hypergraphs is equivalent to certain algebraic properties of the pair $(A(a)_\si, L)$. 

In this paper, we find it useful to discuss \emph{systems of varieties} rather than definable hypergraphs. We say that $\Om$ is a system of varieties on the finite set $V$ over the difference field $A$ if $\Om$ is a functor from the powerset $P(V)$ to difference varieties of finite total dimension such that for every $u \subseteq V$, $\Om(u)$ is a finite cover of the fibre product $\prod(\Om(v), v \in P(u)^-)$. 
See section \ref{subsection: regular systems of varieties} for the precise definitions and notation. When $\Om$ is a system of varieties, we are interested in the definable sets $\rho_u\Om(u)$ obtained by projecting the variety $\Om(u)$ down onto the fibre product $\Om(u)^- = \prod(\Om(v), v\in P(u)^-)$. The definable sets $\rho_u\Om(u)$ are Galois formulas, by definition.

Systems of varieties relate to hypergraphs as follows. An $n$-partite $n$-uniform hypergraph $G$ can be viewed as a functor $G$ on the powerset of $\{1, \ldots, n\}$ together with sets $X_1, \ldots, X_n$ such that for every subset $u$, $G(u) \subseteq \prod_{i \in u} X_i$ and $G(u)$ is contained in the fibre product of the sets $\{G(v) \mid v \subsetneq u\}$. Here the fibre product is defined with respect to the system of projections $\prod_{i \in u} X_i \to \prod_{i \in v}X_i$ for $v \subseteq u$. To say that $G(u)$ is contained in the fibre product of the sets $\{G(v) \mid v \subsetneq u\}$ is equivalent to saying that $G(u)$ is contained in the set $\{(x_i)_{i \in u} \in \prod_{i \in u}X_i \mid (x_i)_{i \in v} \in G(v), v \subsetneq u\}$. Equivalently, $G(u)$ is contained in the set of cliques of the $n$-partite $(n-1)$-uniform hypergraph $(G(v))_{v \subsetneq u}$.  Under this point of view, it is clear how to translate any statement about a definable hypergraph into a statement about a system of varieties. 

However, systems of varieties offer finer control over definable sets, as they allow us to move between the various finite Galois covers inside the system, whereas a definable hypergraph $G$ only gives access to a single Galois cover. As is common in the model theory of pseudofinite fields or ACFA, the correct point of view on definable sets is an \emph{\'etale point of view}. In this paper we will not make any technical use of the term `\'etale', but we will use it to underline the philosophy and the context of our results. Therefore, a system of varieties is really the \'etale analog of a definable hypergraph.

 There is a natural correspondence between systems of varieties and \emph{systems of difference fields}, obtained by taking   the difference function fields associated to the varieties in a system of varieties $\Om$. See section \ref{subsection: regular systems of difference fields} for precise definitions. If $\S$ is a system of difference fields associated to $\Om$,   we will say that $\Om$ is \emph{regular} if for all $u \subseteq V$, $\S(u)$ is linearly disjoint from the   composite   of fields $(\S(v)^{alg})_{v \in P(u)^-}$ over the   the composite $(\S(v))_{v \in P(u)^-}$. This algebraic notion of regularity will be seen to be equivalent to the combinatorial one for hypergraphs. Our choice of terminology  results from a happy coincidence, since our algebraic notion of regularity can be seen as a generalisation of the usual notion of regularity for field extensions. Regularity can also be defined for $\Om$ in purely geometric terms as a certain form of preservation of irreducibility under base change, see Proposition \ref{proposition: geometric regularity}.

\bigskip

 We formulate two algebraic  hypergraph regularity lemmas, one in the \'etale setting and one in the ``classical'' setting. See Theorems \ref{theorem: combinatorial hypergraph regularity with covers} and \ref{theorem: combinatorial regularity, classical partitions} respectively. The \'etale setting provides definable partitions and good error bounds, whereas the classical setting provides the ``expected'' partitions but these are not definable, and the error bounds become weaker. It is an open question whether it is possible to find a ``classical'' algebraic hypergraph regularity lemma where the partitions are definable sets.   See the end of Section \ref{subsection: classical hypergraph regularity} for a discussion.

Our proof of the algebraic regularity lemma relies on the fundamental result of Gowers which establishes the near-equivalence between \emph{edge-uniformity} and \emph{quasirandomness}. Our terminology is in line with \cite{Gowers2006} but either edge-uniformity or quasirandomness are often referred to simply as  \emph{combinatorial regularity} in the literature. These notions are defined in Definitions \ref{definition: quasirandom}, \ref{definition: e-quasirandom} and \ref{definition: combinatorial edge uniformity}. Because of our shift to the \'etale point of view, we also introduce the notion of \emph{\'etale-edge-uniformity}, see Definitions \ref{definition: definable etale edge uniformity} and \ref{definition: combinatorial edge uniformity}. Lemma \ref{lemma: edge uniform same as etale-edge-uniform} shows that edge-uniformity and \'etale-edge-uniformity are equivalent, but the distinction is useful.

 The second ingredient of our proof of the algebraic regularity lemma   is the stochastic independence theorem \ref{theorem: stochastic independence}. 
We will see that quasirandomness follows immediately from the stochastic independence theorem. This theorem says that certain definable sets contained in regular systems behave like independent random variables. From quasirandomness, it is straightforward to recover \'etale-edge-uniformity by the Gowers equivalence, and one can then deduce edge-uniformity by taking sections. This is carried out in Section \ref{subsection: classical hypergraph regularity}.

 From the stochastic independence theorem, we also deduce the stationarity theorem \ref{corollary: stationarity theorem}. 
 A stationarity theorem was the main tool in Tao's proof of the algebraic regularity lemma but the argument for our algebraic regularity lemma  does not rely on the stationarity theorem. Nevertheless, it is a theorem of wider interest. 
 
  Recall that Tao's algebraic regularity lemma  relied on the fact that  the measure of a formula  $\phi(x, a) \wedge \psi(x, b)$ is controlled by the type of $a$ and the type of $b$. As a result, we can construct  definable partitions $Y_1, \ldots, Y_n$, $Z_1, \ldots, Z_k$ such that for any $i \leq n$ and $j \leq k$, $\mu(\phi(x, a) \wedge \psi(x, b))$ is generically constant for $a \in Y_i$ and $b \in Z_j$. It was shown in \cite{Pillay2013} that this result follows from the fact that the formula $\mu_x(\phi(x, y) \wedge \psi(x, z))$ is stable, so that the partitions $(Y_i), (Z_j)$ can be seen to exist by a general stability-theoretic argument. For this reason, we call this a ``stationarity theorem''.
 
 Our stationarity theorem is somewhat stronger than the one in \cite{Tao2012}, even restricted to the intersection of two definable sets. Given $\phi(x, y)$ and $\psi(x, z)$, we find partitions $X_1, \ldots, X_m$, $Y_1, \ldots, Y_n$ and $Z_1, \ldots, Z_p$ such that for any $(i, j, k)$, the definable sets $\phi(x, a)$ and $\psi(x, b)$ behave like independent random variables on the probability space $X_i$, with measure governed by the sets $Y_j$ and $Z_k$ containing $a$ and $b$ respectively. Moreover, the random variables $\phi(x, a)$ and $\psi(x, b)$ are found to be independent in the probabilistic sense from definable sets $\chi(x)$ contained in $X_i(x)$.
  As a corollary, if $\mu(\phi(x, a) \wedge X_i(x)) > 0$ and $\mu(\psi(x, b) \wedge X_i(x)) > 0$, then $\mu(\phi(x, a) \wedge \psi(x, b) \wedge X_i(x)) > 0$. Therefore our stationarity theorem is an amalgamation theorem.
 
In the case of three sets $\phi(x, y, z)$, $\psi(x, y, t)$ and $\chi(x, z, t)$, our stationarity theorem says that after applying some suitable base changes to the domain, we find compatible partitions of the triples $X\times Y \times Z$, $X \times Z \times T$ and  $X \times Y  \times T$  into irreducible varieties such that for a generic triple $a, b, c$, the sets $\phi(x, a, b)$, $\psi(x, a, c)$ and $\chi(x, b, c)$ behave like independent random variables on any domain in the partition. In particular, the measure $\mu(\phi(x, a, b) \wedge \psi(x, a, c) \wedge \chi(x, b, c))$ is uniquely determined by $\tp(a, b)$, $\tp(a, c)$ and $\tp(b, c)$, and this is a direct generalisation of Tao's original result. The reliance on base change is essential though, as we show at the end of Section \ref{subsection: classical hypergraph regularity}.

\bigskip

We conclude this introduction by highlighting some related work concerning Szemer\'edi regularity in model theory.  Model theory has been largely successful in finding settings where  the classical Szemer\'edi graph and hypergraph regularity lemmas can be strengthened. \cite{ChernikovStarchenko2016} shows that the classical hypergraph regularity lemmas can be strengthened for definable hypergraphs in NIP structures. See also \cite{ChernikovTowsner2020} for an extension of those results to functions and for a proof of a strong regularity   lemma when the edge-relation is $\mathrm{NIP}_2$.
 Our results lie at the opposite end of the model theoretic spectrum, since we work here in a  simple unstable theory. 
 
 The recent results of \cite{TerryWolf2021} study hypergraph regularity in finite fields without any definability assumptions on the edge-relations but under some new combinatorial restrictions. These conditions aim to generalise stability in the case of graphs to the setting of hypergraphs. It is not yet clear  if our approach for treating definable hypergraphs can be connected to this combinatorial project.
 
Definable graphs in finite fields coming from difference varieties of finite total dimension have already been considered in \cite{DzamonjaTomasic2017}. In that work, the authors adapt Tao's approach to Szemer\'edi regularity to the definable measure in ACFA over definable sets of finite total dimension, and they also consider graphons. This measure was first shown to be definable in \cite{RytenTomasic2006}, drawing on the twisted Lang-Weil estimates of \cite{HrushovskiFrob}\footnote{The manuscript we reference dates from 2022 but the twisted Lang-Weil estimates were first proved by Hrushovski in the early 2000s.}. \cite{DzamonjaTomasic2017} also adapt the results of Tao on polynomial expansion to setting of varieties of  finite total dimension. In our work, we   use a new approach to derive the algebraic graph and hypergraph regularity lemmas.

Finally, the results of \cite{Tomasic2006} establish   probabilistic independence results for definable sets in pseudofinite fields. This can be seen as a weak version of our stationarity theorem. The approach for that theorem in \cite{Tomasic2006} is quite different from the approach we take here and it is not clear if it yields easily the general Stochastic Independence Theorem which we need for our algebraic hypergraph regularity lemma.

\bigskip

 The paper is structured as follows. In the first section, we revisit classical results about quantifier-elimination in ACFA and we give a new characterisation of the definable measure for definable sets of finite total dimension. 
 
 In the second section, we introduce systems of difference fields and varieties. We define the notion of \emph{definable \'etale-edge-uniformity} and prove a natural hypergraph regularity lemma at the level of the theory ACFA. See Theorem \ref{theorem: definable etale edge uniformity in ACFA}.
 
 In the third section, we prove the stochastic independence theorem and the stationarity theorem. We define quasirandomness at the level of ACFA, and we show that regular systems of varieties are quasirandom. 
 
 In the fourth section, we apply the twisted Lang-Weil estimates to deduce the algebraic hypergraph regularity lemmas. We prove the classical equivalence of Gowers in the \'etale setting and we show how to decompose \'etale systems into disjoint sections to recover more classical regular decompositions of hypergraphs. Our main results are Theorems \ref{theorem: combinatorial hypergraph regularity with covers} and \ref{theorem: combinatorial regularity, classical partitions}.
 
 The results in this paper were arrived at by joint work  that was carried out in two stages. A first algebraic hypergraph regularity lemma was proved by the second named author and appeared in his PhD thesis. The proof used a direct combinatorial argument which relied on a weaker form of the stationarity theorem, and was restricted to the setting of Galois-rigid pseudofinite fields. This setting is discussed after Theorem \ref{theorem: combinatorial regularity, classical partitions}. The current version of this paper is part of the PhD thesis work of the first named author. Both authors are greatly indebted to Ehud Hrushovski, their common PhD supervisor, for his guidance and advice on this project.

\section{The Definable Measure in ACFA}\label{section: definable measure}
 We   study the definable measure on   definable sets in ACFA of finite total dimension. The existence of this measure is implicit in \cite{HrushovskiFrob} and a detailed presentation can be found in \cite{RytenTomasic2006}. In that paper, the authors show that it is possible to use the same kind of quantifier-elimination as is used in \cite{ChatzidakisVanDDMacintyre1992} to construct the definable measure. We will need more fine-grained information about this definable measure    so we find it easier to repeat the presentation from the beginning.

\subsection{Some background and notation}

 \medskip
 
 We use the language of rings with a difference operator $\sigma$.  A difference field is a field $K$ where $\sigma$ is an endomorphism. We say $K$ is inversive if $\sigma$ is an automorphism of $K$. In this paper, we are careful to distinguish the notions of  field, difference field, and  inversive difference fields. If $A$ is a subset of $K$, then $A_\si$ will denote the smallest difference field containing $A$ and $A^{inv}$ will denote the smallest inversive difference field containing $A$.

Let $A \leq K$ be a difference field.  A difference variety in $K$ over $A$ is the solution set of  a finite set of difference equations in $\si$ over $A$.  Note that difference varieties can use positive powers of $\si$ but not negative powers.

 \medskip
 
 \noindent \textbf{Convention:} We will usually discuss difference varieties, so ``variety'' will refer to difference varieties. We will be careful to say ``algebraic variety'' if we want to refer to the  algebraic setting.
 
 \medskip

As in the   algebraic setting, if $X$ is a variety over $A$, we say that $X$ is irreducible if $X$ cannot be expressed as a finite union of proper subvarieties. Note that we do not assume varieties to be irreducible.
When $A$ is an algebraically closed inversive difference field and $X$ is irreducible over $A$, then $X$ is absolutely irreducible, meaning that for every extension $B$ of $A$, $X$ remains irreducible over $B$. See \cite{Cohn1965} or \cite{Levin2008} for detailed presentations of difference algebra.

If $X$ is an irreducible variety over a difference field $A \leq K$, then we can define the difference function field over $A$ associated to $X$ in the usual way. We denote it $A(X)_\si$. Define the \emph{transformal dimension} of $X$ over $A$ to be the maximal cardinality of a set $S \subseteq A(X)_\si$ such that the elements $\si^i(x)$ where $x \in S$ and $i \geq 0$ are algebraically independent over $A$. When $X$ has transformal dimension $0$, we define the \emph{total dimension} of $X$ to be the transcendence degree of $A(X)_\si$ over $A$. We write $dim(X)$ for the total dimension of $X$ over $A$. See section 4.1 of \cite{HrushovskiFrob} for a more complete account of various notions of dimension in difference algebra.

\medskip

Recall that ACFA is the model completion of the theory of inversive difference fields. See \cite{ChatzidakisHrushovski1999} for an explicit axiomatisation.  ACFA is not complete but all completions of ACFA are obtained by specifying the field characteristic and the action of $\si$ on the algebraic closure of the prime field. Recall also that for any $K \models ACFA$ the fixed field $\si(x) = x$ is a model of the theory of pseudofinite fields. The reader is referred to the first few pages of \cite{ChatzidakisHrushovski1999} for the basic model theory of ACFA. In this paper we will revisit quantifier elimination down to existential formulas and the independence theorem.

The fundamental theorem of \cite{HrushovskiFrob} says that ACFA is the asymptotic theory of the difference fields $K_q$, where $q$ is a prime power, $K_q$ is the algebraic closure of the field with $q$ elements and $\si$ is interpreted as $x \mapsto x^q$. We will rely heavily on the twisted Lang-Weil estimates of \cite{HrushovskiFrob}.  In Section \ref{subsection: definable measure basics} we will use some algebraic corollaries of the main theorem of \cite{HrushovskiFrob}, and we will use the full strength of this theorem in Section \ref{section: combinatorics} when we discuss asymptotic counting estimates.

\subsection{Galois formulas} 
We review   quantifier-elimination in ACFA. 
Let $A$ be a  difference field. A finite Galois extension $L$ of $A$ will be said to be invariant if $\sigma$ extends to some endomorphism of $L$. The join of all invariant finite Galois extensions of $A$ will be called the \emph{effective algebraic closure of $A$} and we write $A^{e}$ for this field.
Combining  Corollary 1.5 and (2.8) in \cite{ChatzidakisHrushovski1999} yields the following description of types in ACFA:

\begin{lemma}[\cite{ChatzidakisHrushovski1999}]\label{lemma: description of types}
Let $K \models ACFA$, let $A \leq K$ be an inversive difference field and let $a, b$ be tuples in $K$.  Let $A(a)^{inv}$ and $A(b)^{inv}$ be the inversive closure of the fields $A(a)$ and $A(b)$ respectively. Then $\tp(a / A) = \tp(b / A)$ if and only if there is a difference field isomorphism between $(A(a)^{inv})^{e} \to (A(b)^{inv})^{e}$ fixing $A$ and sending $a$ to $b$.
\end{lemma}

We will deduce a restricted form of quantifier-elimination for $A$-definable sets contained in $A$-definable varieties of finite total dimension. We will rely heavily on the next definition for the rest of this chapter.

\begin{definition}\label{definition: galois formula}
Let $A$ be a difference field. We say that a formula $\phi(x)$ over $A$ is a Galois formula   if    there are $A$-definable varieties  of finite total dimension $V_2(x)$ and  $V_1(x,y)$ where  $ y$ is a  tuple of variables, and a finite family of polynomials $\mathcal{P}$ over $A(x)_\si$ such that, in any algebraically closed field $K$ containing $A$
\begin{enumerate}
\item  the set of realisations of  $\phi$ is equal to the projection $V_1(x, y) \to V_2(x)$   
\item for every generic $(a, b) \in V_1$, the difference field $A(a,b)_\si$ is   the splitting field of   $\mathcal{P}$ over $A(a)_\si$ and is a finite invariant extension of $A(a)_\si$. 
\end{enumerate}

\end{definition}

We do not set an explicit syntax for Galois formulas, but it is clear how to do this if one so wishes. We will assume that we have decided on some explicit syntax for Galois formulas, and we will assume that the family of polynomials $\mathcal{P}$ in Definition \ref{definition: galois formula} is always explicitly given. 

When $\phi$ is a Galois formula over $A$ and  $\mathcal{P}$, $V_1, V_2$  are as in the definition and $a$ is a generic point of $V_2$, we will  say that the splitting field $L$ of $\mathcal{P}$ over $A(a)_\si$ is \emph{the field extension associated to $\phi$}. This terminology applies even if we are in a field $K$ where $a$ does not belong to the set $\phi$. Note that in such a situation, $L$ is always a finite invariant extension of $A(a)_\si$. 

 We will often work with the perfect hull $A(a)_\si^{insep}$, in which case the  \emph{Galois extension associated to $\phi$} is the splitting field of $\mathcal{P}$ over $A(a)_\si^{insep}$.

\begin{lemma}\label{lemma: quanfitier elimination transf dim 0}
Let $K\models ACFA$ and $A\leq K$ an inversive subfield. Let $\phi(x)$ be a definable set over $A$   contained in a variety of  finite total dimension. In $K$, $\phi$ is equivalent   to a Galois formula over $A$.
\end{lemma}

\begin{proof}
We use the following claim:

\begin{claim}
Suppose $a \in K$ lies in a variety over $A$ of finite total dimension. Then $(A(a)_\si^{insep})^e$ is inversive. 
\end{claim}

\begin{proof}[Proof of Claim]
It is enough to prove that $\si^{-1}(a) $ is in $(A(a)_\si^{insep})^e$. Since the extension of $A(a)_\si^{insep}$ by $\si^{-1}(a)$ is clearly invariant, it is enough to show that $\si^{-1}(a)$ is algebraic over $A(a)_\si$. Since $a$ is algebraic over $\si(a), \ldots, \si^n(a)$ for $n$ sufficiently large and $A$ is inversive, the claim follows.
\end{proof}

Observe  that, up to logical equivalence, Galois formulas are preserved under disjunctions. By the claim and Lemma \ref{lemma: description of types}, it is enough to show that for $a, b \in K$ contained in $A$-definable varieties of finite total dimension, $(A(a)_\si^{insep})^e  \cong_A A(b)_\si^{insep})^e$ if and only if $a$ and $b$ satisfy the same Galois formulas over $A$. This is clear. See Proposition 3.7 in \cite{Johnson2019} for a similar proof in a similar context. 
\end{proof}

      When $\phi$ is a Galois formula over $A$, we define the total dimension of $\phi$ over $A$, or $dim(\phi)$, as the total dimension of the variety $V_1$ such that $\phi$ is the projection $V_1 \to V_2$. We also write $cl_A(\phi)$ for the smallest difference variety over $A$ containing $\phi$.

\subsection{The definable measure in ACFA on definable sets of finite total dimension}  \label{subsection: definable measure basics}
 
Let $K \models ACFA$ and let $A \leq K$ be a difference field such that $K$ is $|A|^+$-saturated. The notions of transformal and total dimension clearly extend to quantifier-free definable sets over $A$. For $n \geq 0$, write  $qfDef_n(A) $ for the quantifier-free $A$-definable sets with total dimension $\leq n$. Write $DF_A^n$ for the family of finitely generated difference field extensions of $A$ of transcendence degree $\leq n$. Write $IDF_A^n$ for the family of finitely generated inversive difference field extensions of $A$ of transcendence degree $\leq n$.   There are obvious $1$-$1$-correspondences between $DF_A^n$, $IDF_A^n$, and irreducible difference varieties of dimension $n$ over $A$.
   
 We will see that the transformal function fields provide the correct setting to describe the measure on sets of transformal dimension $0$. We begin by using the main theorem of \cite{HrushovskiFrob} to define the measure on quantifier-free sets of transformal dimension $0$. We reserve the asymptotic component of this theorem for Section \ref{section: combinatorics}, in Theorem \ref{theorem: twisted lang-weil}.

\begin{theorem}[\cite{HrushovskiFrob}]\label{theorem: definable measure on qf sets}
let $K \models ACFA$. For every $n$, there is a definable measure $\mu_n : qfDef_n(K) \to \Qq^{\geq 0}$ such that  
\begin{enumerate}
\item $\mu_n(X) = 0$ if $X$ has total dimension $< n$
\item For every $n$, $\mu_n(\bigwedge_{i = 1}^n \si(x_i) = x_i) = 1$  
\item Fubini holds for the system of measures $(\mu_n)$: if $f : Y \to X$ is a definable surjective map between varieties such that for sufficiently generic $a \in X$, $\mu_n(f^{-1}(a)) = \gamma > 0$ and such that $\mu_k(X) = \lambda > 0$, then $\mu_{n+k}(Y) = \gamma \lambda$. 
\end{enumerate}
\end{theorem} 
 
Fixing $K \models ACFA$ and $A \leq K$, we see that $\mu_n$ on $qfDef_n(A)$ corresponds uniquely to a  function $\mu_n:  DF_A^n \to \Qq^{\geq 0}$ and to a function $IDF_A^n \to \Qq^{\geq 0}$. Indeed, even though an element $L$ in $IDF_A^n$ is usually not finitely generated as a difference field over $A$, it is equal to the inversive hull of some $L \in DF_A^n$ which is uniquely determined up to transformally inseparable extensions. 
 
For $L$ an arbitrary inversive difference field, write $L^{tr}$ for the field $L$ with the automorphism $\sigma^{tr} = \sigma^{-1}$. Then $tr$ is a map $IDF_A^n \to IDF_{A^{tr}}^n$. If $L \in DF_A^n$ and $L' \in DF_{A^{tr}}^n$, we say that $L$ and $L'$ are inversion-dual if $L^{inv} \cong_A ((L')^{inv})^{tr}$, where $L^{inv}$ is the inversive hull of $L$.

We have additional information about the measure $\mu_n$:
\begin{proposition}[\cite{HrushovskiFrob}]\label{proposition: characterisation of measure of difference variety}
Let $K \models ACFA$ and take $(\mu_n)$ the definable measures from Theorem \ref{theorem: definable measure on qf sets}.
\begin{enumerate}
\item For all $n$, $\mu_n$ is invariant under base change in the following sense.  Let $A \leq A'$ be algebraically closed inversive difference fields, let $L \in DF_A^n$ and let $L'$ be the field of fractions of $L \otimes_A A'$. Then $\mu_n(L/A) = \mu_n(L'/A')$.
\item For all $n$, and for all difference fields $A, A' \leq K$, if $A\leq A' \leq A^{inv}$, then $\mu_n(L/A) = \mu_n(L'/A')$ when $L \in DF_A^n$, $L' \in DF_{A'}^n$, and  $L' \leq (L^{insep})^{inv}$, where $L^{insep}$ is the perfect hull of $L$.
\item Let $n \geq 0$, let $A \leq K$ be algebraically closed and inversive and let $L \in DF_A^n$ and $L' \in DF_{A^{tr}}^n$ with $L, L'$ inversion dual. Then 
\[
\mu_n(L/A) = \frac{[L' : \sigma(L')]}{[L : \sigma(L)]_{insep}}.
\]
\end{enumerate}
\end{proposition}

We deduce the following important lemma:

\begin{lemma}\label{lemma: explicit definable measure on varieties}
Let $K \models ACFA$ and let $A \leq K$ be algebraically closed and inversive. Let $n \geq 1$ and let $L, M$ be difference fields with $M/L$ a finite field extension. Then 
\begin{enumerate}
\item if $ L \in DF_A^n$, then  $M \in DF_A^n$ and $[L : \si(L)]=[M : \si(M)]$ and $[L : \si(L)]_{insep} = [M : \si(M)]_{insep}$
\item if $L \in IDF_A^n$, then $M \in IDF_A^n$ and there exist $L_0, M_0 \in DF_A^n$ with $M_0/L_0$ a finite field extension such that $L_0^{inv} = L$ and $M_0^{inv} = M_0L = M$
\item if $M, L \in DF_A^n$, 
there exists $M', L' \in DF_{A^{tr}}^n$ with $M'/L'$ a finite field extension such that $M'$ is  inversion-dual to $M$ and $L'$ is  inversion-dual to $L$.  
\item  if $M, L \in DF_A^n$, $\mu_n(M/A) = \mu_n(L/A)$.\label{equation: explicit definable measure on varieties}
\end{enumerate}
\end{lemma}

\begin{proof}
(1) If $L \in DF_A^n$, it is clear that $M \in DF_A^n$. Moreover, $L / \si(L)$ and $M /\si(M)$ are finite field extensions. Then we have 
$[M : \si(M)][\si(M) : \si(L)] = [M : \si(L)] = [M:L][L : \si(L)]$ and $[\si(M) : \si(L)] = [M : L]$ so $[M : \si(M)] = [L : \si(L)]$.
 Since the same   holds for the separable degree, the statement follows.

(2) If $L\in IDF_A^n$ then $[M : L] = [M :\si(L)] = [M : \si(M)][\si(M) : \si(L)]$ and we deduce that $[M  : \si(M)] = 1$, so $M \in IDF_A^n$. 

  Let $S$ be a finite set of generators of $M$ over $L$. For every $a \in S$ there is polynomial $h_a$ over $L$ such that $\si(a) = h(a)$. Let $F_0$ be a finite set of generators of $L$ over $A$ as an inversive difference field which includes the coefficients of each $h_a$ and of the minimal polynomial of each $a\in S$ over $L$. Define  $L_0 = A(F_0)_\si$ and   $M_0 = L_0(S)_\si$. Then $M_0/L_0$ is a finite extension of fields in $DF_A^n$. Moreover $M_0L$ is inversive so  $M = M_0L$ and hence $M  = M_0^{inv}$.

(3) By (2), $ML^{inv} = M^{inv}$ and we can apply (2) to   $(M^{inv})^{tr} / (L^{inv})^{tr}$ in $IDF_{A^{tr}}^n$.

(4)  By  (1), $[L : \si(L)]_{insep} = [M : \si(M)]_{insep}$. By (3) we can find $M'/L' \in DF_{A^{tr}}^n$ with $M,M'$ and $L, L'$ inversion-dual. By (1), $[L': \si(L')] = [M' : \si(M')]$. By Proposition \ref{proposition: characterisation of measure of difference variety}(3), we have $\mu_n(M/A) = \mu_n(L/A)$.
\end{proof}
   
 By \cite{RytenTomasic2006}, we   know that the measures $\mu_n$ from Theorem \ref{theorem: definable measure on qf sets} extend to definable measures on all definable sets of total dimension $n$ and that these measures have the same properties as in Theorem \ref{theorem: definable measure on qf sets}. Note however that this is straightforward to prove directly from Lemma \ref{lemma: quanfitier elimination transf dim 0}, expressing such definable sets as projections, breaking them up into irreducible components, and counting the degrees of the projections. Henceforth, $\mu_n$ will refer to the definable measures on definable sets of dimension $n$.

\subsection{A new characterisation of the definable measure in ACFA}

In this section, we give a new characterisation of the definable measures $\mu_n$. In fact we could follow this approach from scratch to  define the measures $\mu_n$. As we already have the results of \cite{RytenTomasic2006} or the approach sketched in the previous section,   we will only aim to give an alternative characterisation of the $\mu_n$. 

 If $\phi$ is a definable set in  $K \models ACFA$ and $A \leq K$, recall that $cl_A(\phi)$ is the smallest variety over $A$ containing $\phi$.

 \begin{definition}\label{definition: normalised measure}
Let $A$ be an   inversive difference field. Let $\phi$ be a Galois formula over $A$ and suppose   that    $cl_A(\phi)$ is absolutely irreducible.  Let $a$ be a generic point of $cl_A(\phi)$ and let   $L$ be the field extension of   $A(a)_\si^{insep}$ associated to $\phi$. We  view $L$ as a field with no difference structure.
 
  Define $N(\phi)$ to be the number of extensions of $\sigma$ from $A(a)_\si^{insep}$ to an endomorphism  $\tau$   of $L$ such that $\phi(a)$ is satisfied in the difference field $(L, \tau)$. Define $\nu(\phi ) = \frac{N(\phi)}{[L  : A(a)_\si^{insep}]}$.
 \end{definition}

 Fix $\phi$ and a generic point $a$ of $cl_A(\phi)$ as in the above definition. Choose a partition $(c, b)$ of the tuple $a$ such that $L$ is a regular extension of $A(b)_\si^{insep}$. Write $B = (A(b)_\si)^{alg}$. Define $M(\phi(x, b))$ to be the number of extensions of  $\si$ from $B(a)_\si$ to an endomorphism $\tau$ of the composite $BL $ such that $\phi(a)$ is satisfied in the structure $(BL, \tau)$.

 \begin{lemma}\label{lemma: finite extensions necessary for phi is all we need}
Take $\phi$, $a = (c, b)$ and  $L$ as above and suppose that $L$ is a regular extension of $A(b)_\si^{insep}$. Then  $N(\phi ) = M(\phi(x, b))$. 
 \end{lemma}
 
 \begin{proof}
As before, let $B = (A(b)_\si)^{alg}$.   Let $\Sigma_1$ and $\Sigma_2$ respectively be the sets of extensions of $\si$ from $A(a)^{insep}_\si$ to $L$   and  from $B(a)_\si$ to $BL$ such that  $\phi(a)$ is satisfied in the corresponding structures. We check that restriction to $L$ induces a bijection $\Sigma_2 \to \Sigma_1$. 

Note that $\si$ extends uniquely from $B(a)_\si$ to $B(a)_\si^{insep}$. Since any $\tau \in \Sigma_2$ must be an endomorphism of $L$, it is clear that restriction induces an injection $\Sigma_2 \to \Sigma_1$. Conversely, for $\tau \in \Sigma_1$,   use the fact that $L$ is a regular extension of $A(b)_\si $ to extend $\si$ to an endomorphism $\tau$ of $BL$ compatible with $\si$ on $B$. Then $\phi(a)$ is true in $(BL, \tau)$  because $\phi$ is an existential formula. 
 \end{proof}
 
 Let $K \models ACFA$ and let $A$ be a small inversive subfield. Let $\phi$ be a Galois formula over $A$ and let $a = (c, b)$ , $B$ and $L$  be as above. Observe that even if $\phi(a)$ is false in $K$, $\phi(x, b)$ is satisfiable. 
  Indeed, using model completeness and the description of types in Lemma \ref{lemma: description of types},  we can construct a difference field containing $B$ and realising $\phi(x, b)$ and we can embed this field over $B$ in  $K$.  %Therefore, with $a = (c, b)$ in the variety $V$ as before, the projection from $V$ to the coordinates corresponding to $c$ is generically surjective. 

Informally, the next proposition says that with $a$, $B$ and $L$ as above, the extensions of $\si$ from $B(a)_\si$ to $L$ are all equally likely to arise inside a model $K$ of ACFA and this determines the probability that $\phi(a)$ is true, when the measure is normalised by $cl_A(\phi)$.  We write $cl_A(\phi)$ for the smallest difference variety over $A$ containing $\phi$.
   
   \begin{proposition}\label{proposition: alternative characterisation of measure}
Let $K \models ACFA$ and let $A \leq K$ be an inversive subfield.  Let $\phi$ be a Galois formula over $A$  such that $ cl_A(\phi)$ is  absolutely irreducible over $A$. Let $a \in cl_A(\phi)$ be generic over $A$ and let $L$ be the splitting field associated to $\phi$ over $A(a)_\si^{insep}$. Let $a = (c, b)$ be a partition of $a$ and suppose that $L$ is a regular extension of $A(b)_\si^{insep}$. Write $B = (A(b)_\si)^{alg}$ and suppose that $cl_B(\phi(x, b))$ has total dimension $d$.
    Then 
   \[
   \mu_d(\phi(x, b)) = \nu(\phi )\mu_d(cl_B(\phi(x, b))).
   \]
%       Let  $\Om$ be a protected  $P(E)^-$-system of irreducible varieties over $A$. Let $\phi$ be an invariantly-algebraically-bounded formula over $A$ contained in $\Om(E)$ such that $ cl_A(\phi) = \Om(E)$. Let $u \in P(E)^-$ and suppose that $\Om$ is $u$-adapted to  $\phi$.   Suppose $ \Om(E \setminus u)$ has total dimension $p$.
% 
%Let $b $ be  a generic point of $\Om(u)$.  Then 
%   \[
%   \mu_p(\phi(\Om, b)) = \nu(\phi )\mu_p(\Om(E \setminus u)) .
%   \]
\end{proposition}

\begin{proof}
 %By Lemma \ref{lemma: measures of varieties in cover well behaved}, we can find a generic point $a$ of $\Om(E)$ which maps down onto $b$.   Write $B =  (A(b)_\si)^{alg}$.  Let $L$ be the field extension of $A(a)_\si$ associated to the quantifiers of $\phi$ and write $L' = L^{insep}$. By Lemma \ref{lemma: finite extensions necessary for phi is all we need} and Lemma \ref{lemma: measures of varieties in cover well behaved}, it is enough to show that
This proof is similar to the proof of Proposition 11.1 in \cite{Hrushovski2002}.
Throughout the proof we only consider the definable measure in dimension $d$ so we write $\mu = \mu_d$.   Write $B =  (A(b)_\si)^{alg}$.   By Lemma \ref{lemma: finite extensions necessary for phi is all we need}, it is enough to show that
\[
\mu(\phi(x, b)) = \nu'(\phi(x, b)) \mu(cl_B(\phi(x, b)))
\]
 where 
 \[
 \nu'(\phi(x, b)) = \frac{M(\phi(x,b))}{[BL : B(a)_\si^{insep}]} = \frac{M(\phi(x ,b))}{[L : A(a)_\si^{insep}]}.
 \]

Suppose that $\phi(x, y)$ is the projection of varieties  $\pi: Y  \to X$. 
  Write $X(b), Y(b)$ respectively for the subvarieties of $X,Y$  consisting of points which project onto $b$ by restricting to the appropriate coordinates. Therefore, $\phi(x, b)$ is the image of the projection $\pi : Y(b) \to X(b)$. We can assume that $X(b) = cl_B(\phi(x, b))$. In order to calculate $M(\phi(x, b))$, we can consider each irreducible component of $Y(b)$ over $B$ separately, since these correspond to different extensions of $\si$ from $B(a)_\si$ to $BL$. Therefore we assume that $Y(b)$ is irreducible over $B$.  Then the projection $\pi : Y(b) \to X(b)$ has constant degree $\eta$. 

 Let $\Sigma$ be the set of extensions $\tau $ of $\si$ from $B(a)_\si$ to $BL$ such that $\phi(a)$ is true in $(BL, \tau)$. Write $G = Gal(BL / B(a)_\si^{insep})$. Note  that $G$ acts  on $\Sigma$ by both left and right translation and hence by conjugation. Fix an element $\tau_0$ of $\Sigma$ and let $H = C_G(\tau_0) = \{g \in G \mid \tau_0 g = g \tau_0\}$.

%Let $\Lambda$ be the variety over $A$ such that elements of $\Lambda$ are of the form $(z, t)$ where $z \in cl_A(\phi)$ and $t$ is an enumeration of the distinct roots of the polynomials in the quantifiers of $\phi$.  Write $\pi : \Lambda \to cl_A(\phi)$ for the finite projection and note that   $\pi^{-1}(a)$ generates $L$ over $A(a)_\si^{insep}$. Moreover, $\phi$ determines an $A$-definable difference variety $Y \subseteq \Lambda$ with $\phi = \pi(Y)$. 

% Write $\Lambda(b)$ and $Y(b)$ for the set of elements in $\Lambda$ and $Y$ respectively which project onto $b$ when restricting to the appropriate coordinates. %  By the comment following Lemma \ref{lemma: finite extensions necessary for phi is all we need}, $Y(b)$ has dimension $d$.  In order to calculate $M(\phi( x , b))$, we can consider each irreducible component of $Y(b)$  over $B$ separately, since these correspond to distinct extensions of $\si$ from $B(a)_\si^{insep}$ to $BL$. Therefore we assume that $Y(b)$ is irreducible over $B $. Write $\eta = |\pi^{-1}(a) \cap Y(b)|$. 

\begin{claim}
$\eta = |H|$.
\end{claim}

\begin{proof}[Proof of claim]
We work in $(BL, \tau_0)$ and we show that $H$ acts sharply transitively on the pullback of $a$ to $Y(b)$. Let $y_1, y_2 \in \pi^{-1}(a)$. 
 Since $Y(b)$ is irreducible over $B$,  there is  $g \in G$ such that $g(y_1) = y_2$. Since $BL$ is Galois over $B(a)_\si^{insep}$ and $y_1, y_2$ both generate $BL$ over $B(a)_\si^{insep}$, $g$ is unique. 

We check that $g \in H$. Since $L$ is $\tau_0$-invariant, we have $\tau_0(y_1) = P(y_1)$ for some   polynomial map $G$ over $A(a)_\si^{insep}$. Since  $Y(b)$ is irreducible, we also have $\tau_0(y_2) = P(y_2)$. Now we conclude $\tau_0(g(y_1)) = P(y_2) = g(P(y_1)) = g(\tau_0(y_1))$ so $g \in H$. 
\end{proof}

\begin{claim}
$|\Sigma| = [G : H]$. 
\end{claim}

 \begin{proof}[Proof of claim]
If $\tau_1, \tau_2 \in \Sigma$,   by irreducibility of $Y(b)$ we have $(BL, \tau_1) \cong_{B(a)_\si} (BL, \tau_2)$. Hence there is $g \in G$ with $g\tau_1 = \tau_2 g$. Therefore $\tau_1$ and $\tau_2$ are $G$-conjugate and the number of such conjugates is $[G : H]$.
 \end{proof}
 
\noindent By the first claim and by Fubini, we have 
\[
\mu(\phi(x,b)) = \frac{\mu(Y( b ))}{|H|}.
\] 
By irreducibility of $Y(b)$ and $X(b)$ and   by Lemma \ref{lemma: explicit definable measure on varieties}(\ref{equation: explicit definable measure on varieties}) we also have $\mu(Y(b)) = \mu(X(b))$. By the second claim, we deduce 
 \[
 \mu(\phi(x, b))  = \frac{1}{|H|}   \mu(X(b)) = \frac{|\Sigma|}{|G|}\mu(X(b)) = \nu'(\phi(x, b))\mu(X(b))
 \]
\end{proof}
  
In Proposition \ref{proposition: alternative characterisation of measure}, the assumption that $L$ is a regular extension of $A(b)_\si^{insep}$ is essential. We will usually need to change the domain of definition of $\phi$  in order to put ourselves in the situation where $L$ is regular over $A(b)_\si^{insep}$. Since $L$ is regular over a finite  extension of $A(b)_\si^{insep}$, this can be achieved by changing the domain of $\phi$ and lifting the whole definable set to a cover, so that $\phi$ ranges over tuples of the form $(x, y, y')$ and $y'$ codes Galois information.

Moreover, we will be interested in applying Proposition \ref{proposition: alternative characterisation of measure} to many different partitions of the tuple $a$ at the same time. This motivates the notion of systems of difference fields and difference varieties in the next section.

\section{Regular systems and definable edge-uniformity}

In this section, we often work inside an arbitrary algebraically closed inversive difference field $K$. We introduce systems of difference fields and of varieties, and we isolate the fundamental notion of a regular system. This is a purely algebraic or geometric notion, but we will see that it corresponds to a certain form of combinatorial regularity in the sense of Szemeredi when $K$ is a model of ACFA.

 \subsection{Regular systems of perfect difference fields}\label{subsection: regular systems of difference fields}
 
In the following definition, we record some notation  which we will use often:
\begin{definition}
 $V$ usually denotes a finite set. $P(V)$ is the powerset of $V$ and $P(V)^- = P(V) \setminus\{ V\}$. For $k \leq |V|$, $P_k(V) = \{u \in P(V) \mid |u| \leq k\}$. For $u \subseteq V$ and $0 \leq k \leq |u|$, $P(V, u, k) = \{v \subseteq V \mid |v \cap u| \leq k\}$.  
 
 If $I$ is a collection of subsets of $V$ closed under taking subsets, we say that $I$ is downward closed. If in addition $\bigcup I  = V$, we say that $I$ is a  simplicial complex on $V$. 
 
 If $I$ is downward closed and $u \subseteq V$, define $I_u = \{v \cap u \mid v \in I\}$.  If $I$ is a collection of subsets of $V$, write $\partial I = \{v\mid \exists u \in I, v \subseteq u\}$.

\end{definition}

 We fix an algebraically closed inversive difference field $K$. $K$ is not usually assumed to be a model of ACFA and we will be careful to indicate when we need this assumption.

  In this paper, we write composites of fields as products: if $(K_i :  i\in I)$ is a collection of fields contained in a   field $K$, $\prod_{i \in I} K_i$ denotes the composite of this family. We never write down Cartesian products of fields so this notation is not ambiguous.

If  $A, B, C \subseteq K$, we say that $A$ is independent from  $B$ over $C$ if $((AC)^{inv})^{alg}$ is linearly disjoint from $((BC)^{inv})^{alg}$ over $(C^{inv})^{alg}$ and we write $A \ind_C B$. When $(B_i)_I$ is a family of subsets of $K$ and $A \subseteq K$, we say that the family $(B_i)_I$ is independent over $A$ if for all disjoint $I_1, I_2 \subseteq I$, $\prod_{i \in I_1}AB_i$ is independent from $\prod_{i \in I_2}AB_i$ over $A$.

When $K \models ACFA$, this notion of independence coincides with model theoretic independence. It is useful to recall  from \cite{ChatzidakisHrushovski1999} that ACFA has existence of amalgamation for all orders. This is the Generalised Independence Theorem. We will not use this theorem directly because we will work at the more detailed level of systems of difference fields. 

\begin{definition}\label{definition: system of difference fields}
 Let $V$ be a finite set and let $A \leq K$ be a perfect inversive difference field.   
Let $\S$ be a functor from $P(V)$ to finitely generated perfect   difference field extensions of $A$ of finite total dimension contained in $K$,  where the arrows between elements of $P(V)$ and between the difference field extensions of $A$ are inclusions. 

We say that $\S$ is a system of difference fields on $V$ over $A$ if $\S$ satisfies the following conditions:

 \begin{enumerate}
\item $\mathcal{S}(\emptyset) = A$
\item The family $\{\S\{i\} \mid i \in V\}$ is independent over $A$.
\item For $u \subseteq V$ with $|u| \geq 2$, $\S(u)$ is a finite invariant Galois extension of   $\prod_{ v \in P(u)^-}\S(v)$.
\end{enumerate}

Let $I$ be a simplicial complex on $V$. We say that $\S$ is an $I$-system if for all $u \notin I$, $\S(u) = \prod_{v \in P(u)^-} \S(v)$.
 \end{definition}
 
If $\S$ is an $I$-system on $V$ over $A$, we will often want to restrict $\S$ to smaller simplicial complexes. Let $J \subseteq I$ be downward closed, but not necessarily a simplicial complex. Then we write $\S\upharpoonright J$ for the $J$-system on $\bigcup J$ over $A$ defined by $(\S\upharpoonright J)(u) = \prod_{J_u} \S(v)$ for any $u \subseteq \bigcup J$.  

\begin{definition}\label{definition: refinement of system of difference fields}
Let $A \leq K$ be a perfect inversive difference field and let $\S$ be a system of difference fields on $P(V)$ over $A$ in $K$. Let $K'$ be an algebraically closed inversive difference field containing $K$, let $A'$ be a perfect inversive difference field containing $A$ and let $\S'$ be a system of difference fields on $P(V)$ over $A'$ in $K'$. 

Let $I$ be a downward-closed collection of subsets of $V$. We say that $\S'$ is an $I$-refinement of $\S$ over $A$ if
\begin{enumerate}
\item $A'$ is a finitely generated   extension of $A$
\item for every $u \in I$, $\S'(u)$ is a finite invariant Galois extension of $\S(u)$
\item  for every $u \notin I$,  $\S'(u) = \S(u)\prod_{v \in I_u} \S'(v)$.
\end{enumerate}
%
%Let $A \leq A' \leq K$ where $A$, $A'$ are perfect and inversive and let $\S$, $\S'$ be systems of perfect difference fields on $P(V)$ over $A$ and $A'$ respectively.   Let $I$ be a  downward closed collection of subsets of $V$. 
% We say that $\S'$ is an $I$-refinement of $\S$ if 
%\begin{enumerate}
%\item $A'$ is a finitely generated   extension of $A$
%\item for every $u \in I$, $\S'(u)$ is a finite invariant Galois extension of $\S(u)$
%\item  for every $u \notin I$,  $\S'(u) = \S(u)\prod_{v \in I_u} \S'(v)$.
%\end{enumerate}
\end{definition}

\noindent \textbf{Remarks:} (1) We usually move freely between $K$ and $K'$ when taking refinements, even if we do not explicitly reference the fields. In fact, we often omit the ambient field $K$, taking it to be fixed in our background assumptions.

(2) Let $I$ be a simplicial complex on $V$ and $J \subseteq I$ a downward closed subset. Let $\S$ be an $I$-system on $V$ over $A'$ and let $\S'$ be a $J$-refinement of $\S$ over $A$. Then $\S'$ is an $I$-system on $P(V)$ over $A'$. 

\medskip

 The next definition is fundamental for this paper. 
 
\begin{definition}\label{definition: protected system of difference fields}
Let $\S$ be a system of perfect difference fields on $V$ over $A$. We say that $\S$ is regular if   for every nonempty $u \subseteq V$, $\S(u)$ is linearly disjoint from $\prod_{v \in P(u)^-}\S(v)^{alg}$ over $\prod_{v \in P(u)^-}\S(v)$.
 
\end{definition}

 \noindent \textbf{Remark:}
The notion of a $P(V)^-$-system $\S$ of difference fields over $A$  is closely related to the notion of an amalgamation functor over $A$. See \cite{Hrushovski2006} for a general discussion of amalgamation functors. $\S$ is not exactly the solution of an amalgamation functor on $P(V)^-$ because the fields $\S(v)$ are not algebraically closed, but it is the result of truncating the solution $(\S(v)^{alg})_{v \in P(V)^-}$ down to various subfields.

\medskip 
 
 We will show that systems of difference fields can always be refined to regular systems. We use the following technical lemma:

\begin{lemma}\label{lemma: effective algebraic closure in systems}
Let $\S$ be a   system of   difference fields on $V$ over an  algebraically closed difference field $A$ contained in $K$. Suppose $A$ is existentially closed in   $K$. Let $I$ be a   family of subsets of $V$ and let $u \subseteq V$. Then $\S(V)^e\prod_{v\in I} \S(v)^{alg}$ is linearly disjoint from $\S(u)^{alg}$ over $\S(u)^e \prod_{v \in I_u}\S(v)^{alg}$. 
\end{lemma} 

\begin{proof}
%
% We first prove the lemma under the assumption that the base $A$ is an $\omega_1$-saturated algebraically closed difference field. We will remove this assumption at the end of the proof.
%
%We can  assume that $\S$ is an $\{\{i\} \mid i \in V\}$-system. For every $I \in V$, let $a_i$ be a finite tuple such that $\S(i) = A(a_i)_\si^{insep}$ and write $a_v = (a_i)_{i \in v}$ for any $v \subseteq V$. 
%
%Take $d \in \S(V)^e \prod_I \S(v)^{alg} \cap \S(u)^{alg}$. There is a finitely generated   algebraically closed  difference field $B$ contained in $A$ such that $B(a_u)_\si$ is linearly disjoint from $A$ over $B$ and such that $d \in B(a_V)_\si^e \prod_I B(a_v)_\si^{alg} \cap B(a_u)_\si^{alg}$. 
%
%By $\omega_1$-saturation of $A$, we can find a sequence $(a_{V \setminus u,n}')_{n \geq 0}$ in $A$ such that the difference field $B(a_{V \setminus u})_\si$ is isomorphic to the pure field $B((a_{V \setminus u, n}))$ over $B$ by sending $\si^n(a_{V \setminus u})$ to $a_{V \setminus u, n}'$. Then $B((a_{V \setminus u, n}))$ is linearly disjoint from $B(a_{V \setminus u})_\si$ over $B$ and hence the difference field $B(a_V)_\si$ is isomorphic as a pure field to $B(a_u)_\si B((a_{V \setminus u , n}))$. This isomorphism extends to the respective algebraic closures. 

 We first prove the lemma under the assumption that the base $A$ is   $\omega_1$-saturated. We will remove this assumption at the end of the proof.

We can  assume that $\S$ is an $\{\{i\} \mid i \in V\}$-system. For every $i \in V$, let $a_i$ be a finite tuple such that $\S(i) = A(a_i)_\si^{insep}$ and write $a_v = (a_i)_{i \in v}$ for any $v \subseteq V$. 
Take $d \in \S(V)^e \prod_I \S(v)^{alg} \cap \S(u)^{alg}$. There is a finitely generated   algebraically closed  difference field $B$ contained in $A$ such that $B(a_u)_\si$ is linearly disjoint from $A$ over $B$ and such that $d \in B(a_V)_\si^e \prod_I B(a_v)_\si^{alg} \cap B(a_u)_\si^{alg}$. 

By $\omega_1$-saturation and existential closure of $A$, we can find $a'_{V \setminus u}  \in A$ such that the difference fields $B(a_{V\setminus u})_\si$ and $B(a'_{V \setminus u})_\si$ are isomorphic over $B$. Since $B(a'_{V \setminus u)}) \subseteq A$,  $B(a'_{V \setminus u})_\si$ is linearly disjoint from $B(a_{u})_\si$ over $B$. Hence the difference fields $B(a_V)_\si$ and $B(a_u, a'_{V\setminus u})_\si$ are isomorphic over $B(a_u)_\si$. 

As a pure field, $B(a_V)_\si^e$ depends only on the difference field isomorphism type of $B(a_V)_\si$. Hence the isomorphism $B(a_V)_\si \to  B(a_u, a'_{V\setminus u})_\si$ extends to an isomorphism of pure fields $ B(a_V)_\si^{alg} \to  B(a_u, a'_{V\setminus u})_\si^{alg}$   over $B(a_u)_\si^{alg}$ which restricts to  an isomorphism $ B(a_V)_\si^{e} \to  B(a_u, a'_{V\setminus u})_\si^{e}$ over $B(a_u)_\si^e$.  Since $d \in B(a_u)_\si^{alg}$, the isomorphism constructed above fixes $d$ and we  deduce that  $d \in B(a_u, a'_{V \setminus u})_\si^e \prod_I B(a_{v \cap u}, a'_{v \setminus u})_\si^{alg}$. Since $a'_{V \setminus u} \subseteq A$, we have $d \in\S(u)^e \prod_{I_u} \S(v)^{alg}$ as desired.

Now suppose that $A$ is not $\omega_1$-saturated. Consider the structure $K$ with an additional predicate for $A$. Let $K^*$ be an $\omega_1$-saturated elementary extension and let $A^*$ be the $\omega_1$-saturated algebraically closed difference field extending $A$. Then the independence properties of the tuples $a_V$ are preserved by moving to $K^*$ and $A^*$ is also existentially closed in $K^*$, so we can apply the lemma to the system $\S^*$ defined by lifting the base to $A^*$. 

Take $ d \in \S(V)^e \prod_I \S(v)^{alg} \cap \S(u)^{alg}$. Then $d \in K \cap \S^*(u)^e \prod_{I_u} \S^*(v)^{alg}$. By elementarity, we internalise all  parameters into the base $A$ and we recover the desired result. 
\end{proof}

 The next proposition   will provide the regular partitions in our hypergraph regularity results.  
 
 \begin{proposition}\label{proposition: existence of regular refinements}
  Let $\S$ be an $I$-system of   difference fields on $V$ over $A$. Then $\S$ has a $\partial I$-refinement $\S'$ over some finitely generated   extension $A'$ of $A$ which is regular.  
 \end{proposition}
 
 \begin{proof}
Moving to field extensions if necessary, we can assume that $A$ is algebraically closed, inversive, and existentially closed in $K$. The field $A'$ in the statement is then recovered by choosing finitely many parameters in $A$ over which $\S'$ is defined. We use the following claim:

\begin{claim}
$\S(V)^e \cap \prod_{v \in P(V)^-} \S(v)^{alg} = \prod_{v \in P(V)^-} \S(v)^e$. 
\end{claim}

\begin{proof}[Proof of claim]
We show that for every downward-closed  collection $J$ of subsets of $V$, we have $\S(V)^e \cap \prod_J \S(v)^{alg} = \prod_J \S(v)^e$. We proceed by induction on $|J|$. The case $|J| = 1$ is an instance of Lemma \ref{lemma: effective algebraic closure in systems}. 

Let $J = J' \cup \{w\}$ where $P(w)^- \subseteq J'$ and suppose the claim holds for $J'$. Then we know that $\S(V)^e$ is linearly disjoint from $\S(w)^e\prod_{J'} \S(v)^{alg}$ over $\prod_J \S(v)^e$. By the Towers property,  it is enough to show that $\S(V)^e\prod_{J'} \S(v)^{alg} $ is linearly disjoint from $\prod_J \S(v)^{alg}$ over $\S(w)^e \prod_{J'}\S(v)^{alg}$.

By Lemma \ref{lemma: effective algebraic closure in systems}, $\S(V)^e \prod_{J'} \S(v)^{alg}$ is linearly disjoint from $\S(w)^{alg}$ over $\S(w)^e \prod_{P(w)^-} \S(v)^{alg}$. By the Towers property, $\S(V)^e\prod_{J'} \S(v)^{alg} $ is linearly disjoint from $\S(w)^{alg} \prod_{J'} \S(v)^{alg} = \prod_J \S(v)^{alg}$ over $\S(w)^e \prod_{J'} \S(v)^{alg}$, as desired. 
\end{proof}
 
  Now let $(u_n)$ be an enumeration of $I$ such that for all $n$, $u_n$ is a maximal element of $\{u_m \mid m \geq n\}$. Fix $k \geq 0$ and assume we have constructed a $\partial I$-refinement  $\S'$ of $\S$ such that for all $n < k$, $\S'(u_n) \cap \prod_{P(u_n)^-} \S(v)^{alg} = \prod_{ P(u_n)^-} \S'(v)$. 
  
If $\S'(u_k) \cap \prod_{P(u_k)^-} \S(v)^{alg} \neq \prod_{P(u_k)^-} \S'(v)$, then by applying  the claim to $\S' \upharpoonright P(u_k)$, we construct finite extensions $\S''(v) \subseteq \S'(v)^e$ for every $v \in P(u_k)^-$ such that 
\[
\S'(u_k) \cap \prod_{P(u_k)^-} \S(v)^{alg} \subseteq \prod_{P(u_k)^-} \S''(v).
\]
By making coherent choices of field extensions, we can assume that the extensions $\S''(v)$ define a $P(u_k)^-$-refinement of $\S'$.

  $\S''(u_k)$ is clearly  linearly disjoint from $\prod_{P(u_k)^-} \S(v)^{alg}$ over $\prod_{P(u_k)^-} \S''(v)$.  For $n < k$, $\S'' \upharpoonright P(u_n)$ is a $P(u_n)^-$-refinement of $\S' \upharpoonright P(u_n)$, so $\S''(u)$ is also linearly disjoint from $\prod_{P(u_n)^-} \S(v)^{alg}$ over $\prod_{P(u_n)^-} \S''(v)$.  The proposition follows inductively. 
 \end{proof}

 We will now prove some properties of regular systems which will be useful in the rest of this paper. The key   technical fact behind these results is Lemma \ref{lemma: linear disjointness in systems}. Its proof  is close to the argument which underpins the Generalised Independence Theorem in ACFA (see \cite{ChatzidakisHrushovski1999}).

\begin{lemma}\label{lemma: linear disjointness in systems}
  Let $\S$ be an  $I$-system of perfect difference fields on $V$ over $A$. Fix $ u \subseteq V$ and  take $J$ any  collection of subsets of $V$. Then  $\S(V) \prod_{v \in J \cup I_u} \S(v)^{alg}$ is linearly disjoint from $\S(u)^{alg}$ over $\S(u) \prod_{v \in J_u \cup I_u} \S(v)^{alg}$. 

%Let $A \leq K$ be a perfect inversive field and  let $V$ be a finite set. Let $\S$ be an  $I$-system of perfect difference fields on $V$ over $A$. Fix $ u \subseteq V$ and let $I' = \{v \in I \mid v \not \subseteq u\}$. Take $J$ any nonempty collection of subsets of $V$. Then  $\S(V) \prod_{J } \S(v)^{alg}$ is linearly disjoint from $\S(u)^{alg}$ over $\S(u) \prod_{(J \cup I')_u} \S(v)^{alg}$. 

% the following three statements hold: 
%\begin{enumerate}
%\item \[
%\Big( \S(E) \prod_{v \in J} \S(v)^{alg} \Big)\cap \S(u)^{alg} \subseteq  \prod_{v \in I_u \cup J_u }\S(v)^{alg}  .
%\]
%\item  $\S(u) \prod_{v \in J} \S(v)^{alg}   $ is a regular extension of $\S(u)  \prod_{v \in  J_u} \S(v)^{alg}$
% 
%\item $ \prod_{v \in J} \S(v)^{alg} $ is a regular extension of $\prod_{v \in  J_u} \S(v)^{alg}$.
% 
%\end{enumerate}
%\[
%\Big( \S(E) \prod_{v \in J} \S(v)^{alg} \Big)\cap \S(u)^{alg} \subseteq  \prod_{v \in I_u \cup J_u }\S(v)^{alg}  .
%\]
% 
%\[
%\Big( \S(u) \prod_{v \in J} \S(v)^{alg} \Big) \cap \S(u)^{alg} = \S(u)  \prod_{v \in  J_u} \S(v)^{alg}.
%\]
%
%\[
% \prod_{v \in J} \S(v)^{alg}\text{ is a regular extension of }\prod_{v \in  J_u} \S(v)^{alg} .
%\]
 
\end{lemma}

\begin{proof}
Let $(a_i)$ be a finite collection of elements of $\S(u)^{alg}$ and suppose that $(a_i)$ is linearly dependent over $\S(V) \prod_{J \cup I_u} \S(v)^{alg}$ as witnessed by elements $(b_i)$. We can find
\begin{enumerate}
\item  a tuple $c\subseteq \prod_{I_u} \S(v)^{alg}$
\item tuples $d_v$ where $d_v \subseteq \S(v)^{alg}$ for $v \in J$
\item tuples $e_v$ where $e_v \subseteq \S(v)$ for $v \in I$
\item rational maps $f_i(x,y, z)$ for all $i$

\end{enumerate} 
such that $b_i = f_i(c, (d_v)_J, (e_v)_I)$ for all $i$. 

For every $v \in J$   choose tuples $\alpha_v 
\subseteq \S(v \cap u)$ and $\beta_v \subseteq \S(v \setminus u)$ such that $d_v$ is algebraic over $A(\alpha_v, \beta_v)$. Similarly, for every $v \in I$, choose tuples $\alpha'_v \subseteq \S(u \cap v)$ and $\beta'_v \subseteq \S(u \setminus v)$ such that $e_v$ is algebraic over $A(\alpha'_v, \beta'_v)$.

Then $(a_i)$ satisfies some formula  
\[
\phi((x_i),c, (\alpha_v), (\alpha'_v), (\beta_v), (\beta'_v))
\]
  which says that $(a_i)$ is linearly dependent over $\S(V) \prod_{J \cup I_u} \S(v)^{alg}$. This formula is in the language of   fields.
By elementary stability theory, we find tuples $(\delta_v)$ and $(\delta_v')$ in $A^{alg}$ such that $(a_i)$ satisfies $\phi((x_i),c, (\alpha_v), (\alpha'_v), (\delta_v), (\delta'_v))$ and this entails that $(a_i)$ is linearly dependent over $\S(u) \prod_{J_u \cup I_u} \S(v)^{alg}$
\end{proof}

We now prove some useful properties of regular systems. We fix a   system $\S$ of   difference fields on $V$ over $A$.

\begin{lemma}\label{lemma: regular systems: first linear disjointness lemma}
 Suppose $\S$ is regular. For all  $u \subseteq V$, $\S(u)$ is linearly disjoint from $\prod_{v \in P(V, u, |u|-1)} \S(v)^{alg}$ over $\prod_{v\in P(u)^-} \S(v)$. 
\end{lemma}

\begin{proof}
Let $\S' = \S \upharpoonright P_{|u|-1}(V)$. By Lemma \ref{lemma: linear disjointness in systems}, $\S(u) \prod_{P(u)^-} \S(v)^{alg}$ is linearly disjoint from $\prod_{P(V, u, |u|-1)} \S(v)^{alg}$ over $\S'(u) \prod_{P(u)^-} \S(v)^{alg}= \prod_{P(u)^-} \S(v)^{alg}$. The lemma follows from the definition of regularity and the Towers property.
\end{proof}

\begin{lemma}\label{lemma: regular systems: second linear disjointness lemma}
Suppose $\S$ is regular. Let $I$ be a collection of subsets of $V$. Then $\S(V)$ is linearly disjoint from $\prod_{v \in I} \S(v)^{alg}$ over $\prod_{v\in I} \S(v)$. 

%Let $1 \leq k < |V|$ and let $J$ be a collection of subsets of $V$. Then $\prod_{J \cup P_k(V)} \S(v)$ is linearly disjoint from $\prod_{P_k(V)} \S(v)^{alg}$ over $\prod_{P_k(V)} \S(v)$. 
\end{lemma}

\begin{proof}
We proceed by induction to show that for any downward-closed collection $J$ of subsets of $V$,  $\prod_{J \cup I} \S(v)$ is linearly disjoint from $\prod_{I} \S(v)^{alg}$ over $\prod_I \S(v)$. The case $J = \emptyset$ is trivial, so we assume the claim holds for some $J'$ where $J = J' \cup \{u\}$ and $P(u)^- \subseteq J'$. 

By the Towers property, it is enough to show that $\S(u) \prod_{J' \cup  I} \S(v)$ is linearly disjoint from $\prod_{J'} \S(v)\prod_I \S(v)^{alg}$ over $\prod_{J' \cup I } \S(v)$. By Lemma \ref{lemma: regular systems: first linear disjointness lemma}, $\S(u)$ is linearly disjoint from $\prod_{P(V, u, |u|-1)} \S(v)^{alg}$ over $\prod_{P(u)^-} \S(v)$ and the result follows   by the Towers property.
\end{proof}

\begin{lemma}\label{lemma: regular systems: third linear disjointness lemma}
Suppose $\S$ is regular. Let $u \subseteq V$ and let $I$ be a collection of subsets of $V$. Then $\S(u)$ is linearly disjoint from $\prod_{v \in I} \S(v)^{alg}$ over $\prod_{v \in I_u} \S(v)$. 

 \end{lemma}

\begin{proof}
We can assume that $I$ is a simplicial complex of $V$ and that $u \notin I$. Let $\S'$ be the restriction of $\S$ to $I$, so that $\S'(V)  = \prod_{I}\S(v)$ and $\S'(u) = \prod_{I_u} \S(v)$. Applying Lemma \ref{lemma: linear disjointness in systems}, $\S(u)\prod_{I_u} \S(v)^{alg}$ is linearly disjoint from $  \prod_{I} \S(v)^{alg}$ over $ \prod_{I_u} \S(v)^{alg}$. 

By Lemma  \ref{lemma: regular systems: second linear disjointness lemma}, $\S(u)$ is linearly disjoint from $\prod_{I_u}\S(v)^{alg}$ over $\prod_{I_u} \S(v)$ and the result follows by the Towers property.
%
%Let $\S'$ be the restriction of $\S$ to $P_k(V)$, so that $\S'(V) \prod_{P_k(V)}  \S'(v)^{alg} = \prod_{P_k(V)} \S(v)^{alg}$ and $\S'(u) \prod_{P_k(u)} \S'(v)^{alg} = \prod_{P_k(u)} \S(v)^{alg}$.  Applying Lemma \ref{lemma: linear disjointness in systems}, $\S(u)\prod_{P_k(u)} \S(v)^{alg}$ is linearly disjoint from $  \prod_{P_k(V)} \S(v)^{alg}$ over $ \prod_{P_k(u)} \S(v)^{alg}$. 
%
%
%
%Note that the restriction of $\S$ to $P(u)$ is a regular system on $u$. By Lemma  \ref{lemma: regular systems: second linear disjointness lemma}, $\S(u)$ is linearly disjoint from $\prod_{P_k(u)}\S(v)^{alg}$ over $\prod_{P_k(u)} \S(v)$ and the result follows by the Towers property.
\end{proof}

 \subsection{Regular systems of varieties}\label{subsection: regular systems of varieties}

We introduce systems of varieties. These are closely related to systems of difference fields and in fact all the key theorems in this paper could be formulated purely in terms of systems of difference fields.   However systems of varieties are more natural from a combinatorial point of view.

In what follows, we will work with projections between varieties. If $X, Y$ are difference varieties in affine space, we say that a map $\pi : X \to Y$ is a projection if $X \subseteq K^m$, $Y \subseteq K^n$ and $\pi$ restricts $X$ to certain coordinates in affine space. We say that $\pi$ is dominant if for every top-dimensional component $Z$ of $Y$, there is a top-dimensional component $Z'$ of $X$ such that $\pi$ sends $Z'$ to $Z$ and $\pi(Z')$ is not contained in a subvariety of $Z$. \footnote{\label{footnote:  use of projections} We use projections  because there are a few places in this chapter where it is useful to refer to the implicit syntactical structure of Galois formulas.  By default, we take Galois formulas   to be the result of projecting one variety onto another. Another advantage of working with projections is that it is easy to see when commutative diagrams arise. This will be essential in our arguments.  However all our results generalise easily to systems of varieties with general morphisms of difference varieties instead of projections.}

Let $V$ be a finite set and let $A$ be a perfect inversive difference field. 
 Let $\Om$ be a contravariant functor from $P(V)$ to \emph{difference varieties over $A$ of finite total dimension with   dominant projections $\Om(u) \to \Om(v)$ when $v \subseteq u$}.   We always assume that $\Om(\emptyset) = 0$, viewed as a variety in $K^0$, and that $\Om(u)$ has positive total dimension over $A$ when $u \neq \emptyset$.

For any downward closed collection $I$ of subsets of $V$,  write $\prod(\Om(v), v\in I)$ for the fibre product of the family $(\Om(v))_{v \in I}$.  This is a difference variety with dominant projections to each $\Om(v)$ such that for any $v \subseteq v' \in I$, the projection $\prod(\Om(v), v\in I) \to \Om(v') \to \Om(v)$ equals the projection $\prod(\Om(v), v\in I) \to \Om(v)$. This variety has the usual universal property of fibre products.  

%\begin{definition}\label{definition: systems of varieties}
%Let $\Om$ be a functor on $P(V)$ as above. We say that $\Om$ is a system of   varieties over $A$ if for all $u \subseteq V$ with $|u| \geq 2$,   the points of $\Om(u)$ can be expressed as triples $(a, a',  b)$ such that 
%\begin{enumerate}
%\item $a$ belongs to  the fibre product $\prod(\Om(v), v \in P(u)^-)$
%\item there is a family $\mathcal{P}'$  of purely inseparable algebraic polynomials over $A(x)_\si$ such that   $a'$ is an enumeration of the roots of the polynomials $\mathcal{P}'$ over $A(a)_\si$
%\item  there is a family $\mathcal{P}$ of separable algebraic polynomials over $A(x, x')_\si$ such that   $b$ is  an enumeration of the roots of the polynomials in $\mathcal{P}$ over $A(a, a')_\si$
%\end{enumerate}
%
%
%\end{definition}

\begin{definition}\label{definition: systems of varieties}
Let $\Om$ be a functor on $P(V)$ as above. We say that $\Om$ is a system of   varieties on $V$ over $A$ if for all $u \subseteq V$ with $|u| \geq 2$,  
\begin{enumerate}
%\item $\Om(u)$ is a finite cover of the fibre product $\prod(\Om(v), v \in P(u)^-)$
\item there is a finite family of polynomials $\mathcal{P}_u$ over $A(x)_\si$ such that points of $\Om(u)$ can be expressed as pairs $(a, b)$ with $a \in \prod(\Om(v), v \in P(u)^-)$ and $A(a, b)_\si$ is the splitting field of $\mathcal{P}_u$ over $A(a)_\si$ 
\item the projection $\Om(u) \to \prod(\Om(v), v \in P(u)^-)$ is dominant. Equivalently,  $dim(\Om(u)) = dim(\prod(\Om(v), v \in P(u)^-))$.
\end{enumerate}

When $I \subseteq P(V)$ is a simplicial complex, we say that $\Om$ is an $I$-system of difference varieties if for all $u \notin I$, $\Om(u) = \prod ( \Om(v) , v \in P(u)^-)$ (equivalently, $\mathcal{P}_u$ is empty).

We say that $\Om$ is irreducible if $\Om(V)$ is absolutely irreducible over $A$. 
\end{definition}

\noindent \textbf{Remarks:} (1) If $\Om$ is irreducible, then for all $u$ with $|u|  \geq 2$,  every projection from  $\Om(u)$ to the fibre product $\prod( \Om(v), v \in P(u)^-)$ has generically constant multiplicity. This will be useful for decomposing systems of varieties into disjoint sections in Section \ref{section: combinatorics}.

 (2) The connection between  Galois formulas and systems of varieties is clear: any Galois formula $\phi$ arises as the image of a projection $\Om(u) \to \prod(\Om(v), v \in P(u)^-)$ in some appropriate system of varieties $\Om$.  Hence systems of varieties are a useful framework for studying definable hypergraphs.

(3)  We remarked after Definition \ref{definition: protected system of difference fields} that the concept of a $P(V)^-$-system of difference fields should be viewed as analogous to the \emph{solution} of an amalgamation problem. Accordingly, a $P(V)^-$-system of   varieties can be viewed as being analogous to the amalgamation problem itself. 

 \medskip
 
 We will rely heavily on the notation introduced in the next definition.

 \begin{definition}\label{definition: restrictions inside systems of varieties}
 Let $\Om$ be a system of   varieties on $V$ over $A$. 
\begin{enumerate}
\item  For every $u \subseteq V$ with $|u| \geq 2$, write $\Om(u)^-$ for the fibre product $\prod (\Om(v) , v \in P(u)^-)$. 
\item For every $u  \subseteq V$, write $\rho_u\Om(u)$ for the projection of $\Om(u)$ onto $\Om(u)^-$. 
\end{enumerate} 
 
 We define analogously $\S(u)^- = \prod_{v \in P(u)^-} \S(v)$ when $\S$ is a system of difference fields on $V$ over $A$. 
 
 \end{definition}  

\begin{definition}\label{definition: refinement of system of varieties}
Let $\Om$   be a system of varieties on $V$ over $A$. Let $A' $ be a perfect inversive  difference field containing $A$ and $\Om'$ a system of varieties on $V$ over $A'$.

We say that $\Om'$ is a refinement of $\Om$ if for every $u \subseteq V$ there is a   finite dominant projection $\Om'(u) \to \Om(u)$ and for $u \subseteq v$, the projections $\Om'(u) \to \Om(u) \to \Om(v)$ and $\Om'(u) \to \Om'(v) \to \Om(v)$ commute.

 Let $I$ be a   collection of subsets of $V$.  We say that $\Om'$ is an $I$-refinement of $\Om$ if $\Om'$ is a refinement and  for every $u \notin I$, $\Om'(u)$ is the fibre product of the varieties $\Om(u), (\Om'(v))_{v \in I_u}$.

 We say that $\Om'$ is a surjective refinement of $\Om$ if the projection $\Om'(V) \to \Om(V)$ is generically surjective.
\end{definition}

\noindent \textbf{Remark:}  Taking an    extension of $A$ if necessary, every system of varieties over $A$ admits a partition  into irreducible systems. If $\Om'$ is a surjective refinement of $\Om$ over the extension $A'$, the  partition of $\Om'$ into irreducible components will give us a notion of   \'etale-partition of $\Om$.

\begin{definition}\label{definition: regular systems of varieties}
Let $\Om$ be a system of varieties on $V$ over $A$. Let $a$ be a generic point of $\Om(V)$ and write $a_u$ for the image of $a$ under the projection $\Om(V) \to \Om(u)$. We define the system of   difference fields associated to $a$ to be the system   $\S(u) = A(a_u)_\si^{insep}$.

 We also say that the system $\S$ above is a system of difference fields associated to $\Om$. When $\Om$ is irreducible, systems of difference fields associated to $\Om$ are unique up to isomorphism.

 We say that $\Om$ is a regular system of varieties  if $\Om$ is irreducible and the system of difference fields associated to $\Om$ is regular.
\end{definition}

% We say that $\Om$ is a \emph{Galois system of difference varieties} if every associated system of difference fields is Galois.

% If $\Om$ is a system of difference varieties, we say that $\Om$ is \emph{surjective} if the maps $\Om(u) \to \Om(v)$ are generically surjective. 

% If $\Om$ is a system of difference varieties over $A$ on $P(E)$, then we can consider the top-dimensional irreducible components of each $\Om(u)$ and this determines a partition of the top-dimensional components of $\Om(E)$. This is how we will obtain our partitions of hypergraphs in the following sections. 

\begin{proposition}\label{proposition: regular refinement of system of varieties}
  Let $\Om$ be an   $I$-system of     varieties over $A$. Then there is a surjective $\partial I$-refinement $\Om'$ of $\Om$ over an extension $A'$ of $A$ such that each irreducible component of $\Om'$  is regular.
\end{proposition}

\begin{proof}
We can assume that $\Om$ is irreducible over $A$. Let $a$ be a generic point of $\Om(V)$ over $A$ and let $\S$ be the associated system of   difference fields on $V$ over $A$. Write $a_u$ for the projection of $a$ to $\Om(u)$. 

 Let $\S'$ be a regular $\partial I$-refinement of $\S$ over a finitely generated extension $A'$ of $A$, as given by Proposition \ref{proposition: existence of regular refinements}. For every $u \in \partial I$, let $\mathcal{P}_u$ be a family of   polynomials over $\S(u)$ such that $\S'(u)$ is the splitting field of $\mathcal{P}_u$ over $\S(u)$. We can assume that $\mathcal{P}_u \subseteq \mathcal{P}_{v}$ when $u \subseteq v$ and we can take $\mathcal{P}_u$ to be over $A'(a_u)_\si$. 
 
We can define a system $\Om'$ over $A'$ refining $\Om$ such that  the projection $\Om'(V) \to \Om(V)$ is generically surjective and for every $u \in \partial I$, $\Om'(u)$ is a variety of points of the form $(c, d)$ where $c \in \Om(u)$ and $A'(c, d)_\si$ is the splitting field of the polynomials $\mathcal{P}_u$ over $A'(c)_\si$.
 
% Then for every $u \in \partial I$ we define $\Om_1(u)$ to be a   variety of points of the form $(x,y)$ where $x \in \Om(u)$, and  for any $(c, d) \in \Om_1(u)$, $A'(c, d)_\si$ is the splitting field of the polynomials $\mathcal{P}(u)$ over $A'(c)_\si$, and the projection $\Om_1(u) \to \Om(u)$ is generically surjective.

  % $x'$ is a collection of roots of purely inseparable polynomials over $A(x)_\si$  and $y$ is an enumeration of the roots of $\mathcal{P}_u$ over $A(x, x')_\si$. 
 
Every irreducible component of  $\Om'$ is regular: for any generic point $b$ of $\Om'(V)$, even though the system of difference fields associated to $b$ may not be  isomorphic to $\S'$ as a system of difference fields, it is isomorphic to $\S'$ in the sense of pure fields, since the field extensions are just the splitting fields of the families $\mathcal{P}_v$. This guarantees regularity of the system associated to $b$. 
 \end{proof}
  
 %\noindent \textbf{Remark:} We can extend the  proof of Proposition \ref{proposition: regular refinement of system of varieties} to show that we can construct $\Om'$ so that the projections $\Om'(u) \to \Om(u)$ have generically constant multiplicity when $|u| \geq 2$ and so that the projections $\Om'(u) \to \prod(\Om'(v), v \in P(u)^-)$ are generically surjective with generically constant multiplicity. This is because $\Om'$ is constructed by simply adding roots of certain polynomials to the varieties of $\Om$.   This will be useful in Lemma \ref{lemma: sections}.
 
 \medskip

\noindent \textbf{Example:} We give a basic example of a $P(V)$-system $\Om$ with a regular $P(V)^-$-refinement. Working over the prime field of $K$ and taking $V = \{0, 1\}$, let $\Om(0)$ and $\Om(1)$ be copies of the fixed field.  Let $\Om(V)$ be the variety of points of the form $(x, y, z, t)$ where $x \in \Om(0)$, $y \in \Om(1)$, $z$ is a square root of $y$ with $\si(z) = z$ and $t$ is a square root of $x+z$ with $\si(t) = t$. Take $\Om'(0) = \Om(0)$, $\Om'(1)$ the variety of points of the form $(y, z)$ where $z$ is a square root of $y$ with $\si(z) = z$, and take $\Om'(V) = \Om(V)$. Then $\Om'$ is a regular $P(V)^-$-refinement of $\Om$. Note that the projection $\Om'(V) \to \Om(V)$ is surjective (in fact it is the identity) and the the projection $\Om'(1) \to \Om(1)$ is only dominant. This is consistent with Definition \ref{definition: refinement of system of varieties}. 

\medskip

While the definition of regularity for systems of difference fields and varieties is natural and easy to state, we note that it is slightly stronger than what will be needed in the rest of this paper. We make this precise in  Proposition \ref{proposition: geometric regularity}. First, we recall some classical notions of difference algebra. If $K$ is a difference field, recall that a difference field extension $L/K$ is \emph{monadic} if for any difference field extension $M/K$, there is at most one difference homomorphism $L \to M$ over $K$. Equivalently, if $L/K$ is Galois, then $L/K$ is monadic if and only if $\{g \in Gal(L/K) \mid g \si = \si g\text{ on }L\} = \{e\}$.

 Recall also that if $L/K$ is a finite Galois extension and is monadic, then $L/K$ is compatible with every difference field extension of $K$, meaning that for any $M/K$, there is $N/K$ and difference homomorphisms $L \to N$ and $M \to N$ over $K$. It follows that if $L/K$ is a finite Galois extension and is monadic, then $\si$ is the unique extension of $\si \upharpoonright K$ to $L$. See \cite{Cohn1965} or \cite{Levin2008} for details.

Let $\phi$ be a Galois formula over an inversive difference field $A$ corresponding to a projection of irreducible varieties $X \to Y$. Let $a$ be a generic point of $Y$ and $L$ the Galois extension of $A(a)_\si^{insep}$ associated to $\phi$.  By inspecting the Galois information present in the proof of Proposition \ref{proposition: alternative characterisation of measure}, we see that if $L/A(a)_\si^{insep}$   gives a monadic extension, then the projection $X \to Y$ has multiplicity $1$ and hence $\mu(\phi) = \mu(Y)$. Therefore, from the geometric or combinatorial point of view, $\phi$ is trivial and information related to monadic field extensions can be ignored. However,  in Definitions \ref{definition: protected system of difference fields} and \ref{definition: regular systems of varieties} we have defined regularity with respect to the effective algebraic closure of difference fields, which means that   our notion of regularity is stronger than what we need in this paper.  The next proposition makes this precise.

\begin{proposition}\label{proposition: geometric regularity}
Let $\Om$ be an irreducible  system of   varieties on $V$ over $A$.  The following are equivalent:
\begin{enumerate}
\item for any system $\S$ of difference fields associated to $\Om$, for every $u \subseteq V$, $\S(u) \cap \prod_{P(u)^-} \S(v)^{alg}$ is a monadic extension of $\S(u)^-$
\item for any $u \subseteq V$,  let $\Om'$ be an irreducible $P(u)^-$-refinement of $\Om \upharpoonright P(u)^-$. Let $\Om^+(u)$ be the fibre product of $\Om(u)$ and $\Om'(u)$ over $\Om(u)^-$. Then $dim(\Om^+(u)) = dim(\Om(u))$. \label{equation: geometric regularity}
\end{enumerate}
 \end{proposition}

\begin{proof}
$(1) \Rightarrow (2)$: We check (2) for $u = V$. Let $\Om'$ and $\Om^+$ be as in the statement. Let $\S$ be a system of difference fields associated to $\Om$ and $\S'$ a system of difference fields associated to $\Om'$. Write $\S''$ for the restriction of $\S'$ to a system associated to $\Om \upharpoonright P(V)^-$.  By irreducibility of $\Om$, there is a difference field isomorphism $f : \S''(V) \to \S(V)^-$. We can extend $f$ to a field homomorphism $f : \S'(V) \to \S(V)^{alg}$. Then $f(\S'(V)) \subseteq \prod \S(v)^{alg}$ so $f(\S'(V))$ is linearly disjoint from $\S(V)$ over $f(\S'(V)) \cap \S(V)$ and $f(\S'(V)) \cap \S(V)/\S(V)^-$ is a monadic extension.

By universal compatibility of monadic extensions and taking a Galois conjugate of $f$ if necessary, we can assume that $f$ is a difference field homomorphism on $f^{-1}(f(\S'(V)) \cap \S(V))$. Then by linear disjointness, we can construct a difference operator $\tau$ on $\S(V)^{alg}$ such that $\tau = \si$ on $\S(V)$ and $\tau = f \circ \si \circ f^{-1}$ on $f(\S'(V))$. Now $(\S(V)f(\S'(V)), \tau)$ is a difference field associated to $\Om^+(V)$ and hence $dim(\Om^+(V)) = dim(\Om(V)) = dim(\Om'(V))$.

$(2) \Rightarrow (1)$: Suppose (1) fails. Let $\S$ be a system of difference fields associated to $\Om$ and assume that $L:= \S(V) \cap \prod_{P(V)^-} \S(v)^{alg}$ is not monadic over $\S(V)^-$. By properties of monadic extensions, there is an extension $\tau$ of $\si \upharpoonright \S(V)^-$ to $L$ such that $(L, \tau)$ is not isomorphic to $(L, \si)$.

 We can find an irreducible  $P(V)^-$-refinement $\Om'$ of $\Om \upharpoonright P(V)^-$  such that for any system $\S'$ associated to $\Om'$, writing $\S''$ for the subfield of $\S'(V)$ corresponding to a generic point of $\Om(V)^-$, $\S'(V)/\S''(V)$ contains a subextension $L'/\S''(V)$ isomorphic to $(L, \tau) / \S(V)^-$. It is clear that the fibre product of $\Om'$ and $\Om$ has dimension smaller than $dim(\Om') = dim(\Om)$, as $\S'(V)$ and $\S(V)$ are incompatible.
\end{proof}

\noindent \textbf{Remarks:} (1) In the remainder of this paper, we will be interested in geometric and combinatorial properties of systems of varieties, so property (\ref{equation: geometric regularity}) in Proposition \ref{proposition: geometric regularity} will be sufficient in all applications of regularity. However, we believe that the notion of regularity given in Definition \ref{definition: protected system of difference fields} and \ref{definition: regular systems of varieties} is more natural, so we will refer to that one for simplicity.

(2) Take $\Om$, $\Om'$, $\Om(u)^-$, $\Om^+(u)$ as in Proposition \ref{proposition: geometric regularity}(\ref{equation: geometric regularity}) and suppose $K \models ACFA$.  Writing $d  = dim(\Om(u)^-)$, we find that if $\Om$   regular, then  $\mu_d(\Om(u)) = \mu_d(\Om^+(u))$ for every $u \subseteq V$. This   property is essentially the notion of  definable \'etale-edge-uniformity which we will study in the next section.

 \medskip

We prove some technical lemmas about regular systems of varieties. The next two lemmas show that systems of difference varieties behave as expected with respect to the definable measure. We fix $V$ a finite set and $I$ an abstract simplicial complex on $V$.

\begin{lemma}\label{lemma: measure of system of varieties is measure of product}
Suppose that $K \models ACFA$ and let $A \leq K$ be perfect and inversive. Let $\Om$ be a regular $I$-system of   varieties on $V$ over $A$ where $V \notin I$. For $i \in V$, write $d_i = dim(\Om(i))$.

Then $\prod(\Om(v), v \in P(V)^-)$ is irreducible, so that $\Om(V) = \prod(\Om(v), v \in P(V)^-)$.
 $\Om(V)$ has dimension $d = \sum_{i \in V} d_i$, and $\mu_d(\Om(V)) = \prod_{i \in V} \mu_{d_i}(\Om(i))$. 
 \end{lemma}

\begin{proof}
 We argue inductively on $|I|$, the base case being $I = \{\{i \} \mid i \in V\}$, which is trivial. Write $I = I' \cup \{u\}$ where $u$ is a maximal element of $I$, so that $P(u)^- \subseteq I'$. 
 Let $\Om'   = \Om \upharpoonright I'$. Then $\Om'$ is a  regular system.   We   assume that the lemma holds for $\Om'$ and $\Om\upharpoonright P(u)^-$.
 
 We argue that $X: = \prod(\Om(v), v \in P(V)^-)$ is irreducible, which will entail that $\Om(V) = X$. Let $x_0, x_1$ be  generic points of $X$ in $K$. We need to show that $A(x_0)_\si$ is isomorphic to $A(x_1)_\si$ over $A$ as a difference field. Write $x'_0, x'_1$ for the projections of $x_0, x_1$ to $\Om'(V)$. Let $\S_0, \S_1, \S_0', \S_1'$ be the systems of difference fields associated to $x_0, x_1, x_0', x_1'$ respectively. 
 
 By induction hypothesis, $\prod_{P(u)^-} \S_0(v)$ is isomorphic to $\prod_{P(u)^-} \S_1(v)$ over $A$, and so are $\S_0'(V)$ and $\S_1'(V)$. $\Om(u)$ is  irreducible  so it is enough to show that $\S_0(u)$ is linearly disjoint from $\S'_0(V)$ over $\prod_{P(u)^-} \S_0(v)$. 
By Lemma  \ref{lemma: regular systems: second linear disjointness lemma}, $\S_0(u)$ is linearly disjoint from $\prod_{P(u)^-} \S_0(v)^{alg}$ over $\prod_{P(u)^-} \S_0(v)$. Therefore it is enough to show that $\S_0(u) \prod_{P(u)^-} \S_0(v)^{alg}$ is linearly disjoint from $\S_0'(V) \prod_{P(u)^-} \S_0(v)^{alg}$. This is a direct application of Lemma \ref{lemma: linear disjointness in systems}.
 
 Hence $\Om(V) = X$. The statements about $dim(\Om(V))$ and $\mu_d(\Om(V))$ follow inductively by direct applications of Lemma \ref{lemma: explicit definable measure on varieties}(\ref{equation: explicit definable measure on varieties}). 
 \end{proof}

\begin{lemma}\label{lemma: measures of varieties in cover well behaved}
Suppose that $K \models ACFA$ and let $A \leq K$ be perfect and inversive. Let $\Om$ be a regular   $I$-system of     varieties on $V$ over $A$. Fix $u \subseteq V$ and for any $b \in \Om(u)$ write $\Om(V, b)$ for the pullback of $b$ to $\Om(V)$. 
Then for any generic $b \in \Om(u)$, $\Om(V, b)$ is an irreducible variety over $(A(b)_\si)^{alg}$ of dimension $d = dim(\Om(V \setminus u))$ and $\mu_d(\Om(V, b)) = \mu_d(\Om(V \setminus u))$. 
 %Fix $u \in P(E)^-$  and let $b \in \Om(u)$ be generic over $A$. Then $C(\Om, b)$ is an irreducible variety  over $(A(b)_\si)^{alg}$ of dimension $p =   dim(\Om(E \setminus u))$ and $\mu_p(C (\Om, b)) = \mu_p(\prod_{j \in E \setminus u}  \Om\{j\})$. 
\end{lemma}

\begin{proof}
   First we show that $\Om(V, b)$ has dimension $d$. Let $a$ be a generic point of $\Om(V)$ and let $\S$ be the associated  system of perfect difference fields.   By Lemma  \ref{lemma: regular systems: second linear disjointness lemma}, $\S(V)$ is linearly disjoint from $\S(u)^{alg}$ over $\S(u)$. Since $\Om(u)$ is irreducible, $\S(u)$ is isomorphic to $A(b)_\si^{insep}$ over $A$.

We define a new automorphism $\tau$ of $ \S(V)^{alg} $ as follows. On $\S(V)$, $\tau$ is equal to $\sigma$. Let $f : \S(u)^{alg} \to A(b)_\si^{alg} $ be an isomorphism of pure fields extending the difference fied isomorphism $\S(u) \to A(b)_\si^{insep}$. Then define $\tau$ on $\S(u)^{alg}$ to be $f^{-1} \circ \si \circ f$. %By linear disjointness, this is compatible with $\si$ on $\S(u)$.
 Now extend $\tau$ arbitrarily to an automorphism of $\S(V)^{alg}$. 

As $\S(V)$ has finite total dimension over $A$, $\S(V)^{alg}$ is inversive. By model completeness, we can find an embedding of $(\S(V)^{alg}, \tau)$ in our model $K$ of ACFA. By the description of types of  Lemma \ref{lemma: description of types}, we know that the copy of $b$ in $\S(V)^{alg}$ satisfies the same complete type as $b$ over $A$. It follows that we can embed $(\S(V)^{alg}, \tau)$ in $K$ over $(A(b)_\si)^{alg}$ and hence $\Om(V, b)$ has dimension $d$.  

Now let $a_1, a_2$ be generic points of $\Om(V, b)$ and let $\S_1, \S_2$ be the corresponding systems of   difference fields. Note that $\S_1(u) = \S_2(u) = (A(b)_\si)^{insep}$.  Write $B = (A(b)_\si)^{alg}$.  To show that $\Om(V, b)$ is irreducible over $B$, it is enough to show that $B(a_1)_\si \cong_{B} B(a_2)_\si$. By regularity, $\S_i(V)$ is linearly disjoint from $B$ over $\S_i(u)$ for $i = 1, 2$. By irreducibility, we have $\S_1(V) \cong_{\S_1(u)} \S_2(V)$. It follows that we can construct a difference field isomorphism $B(a_1)_\si \to B(a_2)_\si $ over $B$ so $\Om(V, b)$ is irreducible. 

The statement about $\mu_d$ follows from the dimension of $\Om(V, b)$ being $d$, $\Om(V, b)$ being irreducible, and Lemma  \ref{lemma: explicit definable measure on varieties}(\ref{equation: explicit definable measure on varieties}).
\end{proof}

\subsection{Definable \'etale-edge-uniformity and a first hypergraph regularity lemma}

\textit{In this section, we work inside a   model $K $ of ACFA and we fix  $A \leq K$ a finitely generated   perfect  inversive difference subfield. $V$ is a finite set.}

\medskip

In this section, we return to the study of Galois formulas and we show that our lemmas about regular systems of difference fields already prove an interesting form of hypergraph regularity: given a definable hypergraph  $\phi$ with edges indexed by the simplicial complex $P(V)^-$, we find an \'etale  partition of the domain of $\phi$ such that each induced sub-hypergraph is edge-uniform with respect to definable $P(V)^-$-refinements   of   $\phi$. We make our terminology precise in the following definition.

\begin{definition}\label{definition: definable etale edge uniformity}
Let $\Om$ be a system of   varieties on $V$ over $A$. For every $u \subseteq V$, write $d_u = dim(\Om(u))$ and write $\phi_u = \rho_u\Om(u)$ for the projection of $\Om(u)$ to $\Om(u)^-$. 

We say that $\Om$ is definably \'etale-edge-uniform if for every $u \subseteq V$ with $|u| \geq 2$, the following holds:
let $\Om'$ be a $P(u)^-$-refinement of $\Om \upharpoonright P(u)^-$ and write $\pi$ for the projection $\Om'(u) \to \Om(u)^-$. Then

 \[
\frac{\mu_{d_u}(\phi_u \wedge \pi(\Om'(u)) )}{\mu_{d_u}(\phi_u)} = \frac{\mu_{d_u}(\pi(\Om'(u)) )}{\mu_{d_u}(\Om(u)^-)}.
\]
\end{definition}

\noindent \textbf{Remarks:} (1) With the notation above, we emphasise that $\Om'(u)$  is the fibre product $\prod(\Om'(v), v \in P(u)^-)$ and that $\pi(\Om'(u))$ is contained in $\Om(u)^-$. Therefore $\Om'(u)$ lives inside the boundary of $\Om(u)$, modulo moving to an \'etale cover. 

(2) If $\Om$ is an $I$-system of varieties, then it is enough to check the property of definition \ref{definition: definable etale edge uniformity} for $u \in I$.

   \medskip

 The following proposition is an elaboration of the comment following Proposition \ref{proposition: geometric regularity}. 

\begin{proposition}\label{proposition: definable edge uniformity and regular systems}
 Let $\Om$ be  a   system of varieties on $V$ over $A$. Then $\Om$ is definably \'etale-edge-uniform if and only if the irreducible components of $\Om$ satisfy either of the equivalent conditions   in Proposition \ref{proposition: geometric regularity}. In particular, if $\Om$ is regular, then $\Om$ is definably \'etale-edge-uniform.
\end{proposition}

\begin{proof}
We can assume that $\Om$ is irreducible. Suppose that $\Om$ satisfies (1) in Proposition \ref{proposition: geometric regularity}. Let $\phi = \rho_V \Om(V)$ and let $\Om'$ be a $P(V)^-$-refinement of $\Om \upharpoonright  P(V)^-$. Write $\pi$ for the projection $ \Om' \to \Om$. 
We need to check that  
\[
\frac{\mu_{d_V}(\phi \wedge \pi(\Om'(V)))}{\mu_{d_V}(\phi)} = \frac{\mu_{d_V}(\pi(\Om'(V)))}{\mu_{d_V}(\Om(V)^-)}.
\]
 
Let $a$ be a generic point of $\Om(V)$, let $\S$ be the associated system of difference fields.  Let   $L'$ be the Galois extension of $\S(V)^-$ associated  to $\Om'(V)$.  
 In the notation of Proposition \ref{proposition: alternative characterisation of measure}, we only need to check that 
\[
N(\phi \wedge \pi(\Om'(V))) = N(\phi)  N(\pi(\Om'(V))). 
\]
 Since $\S(V) \cap L'/\S(V)^-$ is monadic, $N(\phi)$ is equal to the number of extensions $\tau$ of $\si$ from $\S(V) \cap L'$ to $\S(V)$ such that $\phi$ is satisfied in $(L', \tau)$. Similar equalities hold for $N(\pi(\Om'(V))$ and $N(\phi \wedge \pi(\Om'(V)))$. Now $L'$ and $\S(V)$ are linearly disjoint over $\S(V) \cap L'$, so the result follows by counting extensions of $\si$ in the obvious way.

Conversely, suppose there is a $P(V)^-$-refinement $\Om'$ of $\Om \upharpoonright P(V)^-$ such that (\ref{equation: geometric regularity}) of Proposition \ref{proposition: geometric regularity} fails. Then $\phi \cap \pi(\Om'(V))$ is the projection to $\Om(V)^-$ of the fibre product of $\Om'(V)$ and $\Om(V)$ over $\Om(V)^-$. Since $\mu_{d_V}(\Om'(V) \times_{\Om(V)^-} \Om(V)) = 0$, it is clear that definable \'etale-edge-uniformity fails.  
\end{proof}

We state a first hypergraph regularity lemma at the level of ACFA, in terms of systems of varieties and definable edge-uniformity. We repeat all our notation and background assumptions.

\begin{theorem}\label{theorem: definable etale edge uniformity in ACFA}
Let $V$ be a finite set and $I$ a simplicial complex on $V$. Let $A$ be an inversive difference field and let $\Om $ be an $I$-system of varieties on $V$ over $A$.

  Then there is a surjective   $\partial I$-refinement $\Om'$ of $\Om$ over a finitely generated extension $A'$ of $A$   such that each irreducible component of $\Om'$ is definably \'etale-edge-uniform.
\end{theorem}

\begin{proof}
We obtain $\Om'$ by Proposition \ref{proposition: regular refinement of system of varieties}. Definable \'etale-edge-uniformity is given by Proposition \ref{proposition: definable edge uniformity and regular systems}. 
\end{proof}

\section{The Stochastic independence theorem,  the stationarity theorem, and   quasirandomness }

\subsection{The Stochastic Independence Theorem}

\emph{We fix  $K \models ACFA$ and $A \leq K$ a perfect inversive subfield. $V$ is a finite set. } 
 
 \medskip

 The next definition sets up some useful notation for moving between \'etale covers. 
  
\begin{definition}
  Let $\Om$ be a system of difference varieties on $V$ over $A$. Let $\phi$ be a definable set contained in   $\Om(u)$ for some $u \subseteq V$.
  \begin{enumerate}
  \item  If $v \subseteq u$ and  $b \in \Om(v)$, then we write $\phi(\Om(u), b)$ for the set of elements $x \in   \phi$ which map to $b$ under the projection $\Om(u) \to \Om(v)$.
  \item When $u \subseteq w \subseteq V$, we write $\phi(\Om(w))$ for the pullback of $\phi$ under the projection $\Om(w) \to \Om(u)$.
  \item When $\Om'$ is a refinement of $\Om$, we also write $\phi(\Om'(u))$ for the pullback of $\phi$ to $\Om'(u)$ by the  projection $\Om'  \to \Om$. 
\end{enumerate}     

\end{definition}

 In the next theorem, an \emph{antichain $J$ on $V$} is a collection of subsets of $V$  of size at least $2$  such that for any  two $u \neq v \in J$, $u$ does not contain $v$. We ask that elements of $J$ have size at least $2$ because it makes the statement of Theorem \ref{theorem: stochastic independence} easier and there will be results later on which break down if we allow for antichains with singletons.

\begin{theorem}[The Stochastic Independence Theorem]\label{theorem: stochastic independence}
Let $J$ be an antichain of $V$ with $\bigcup J  = V$. Let $I$ be the simplicial complex on $V$ generated by $J$ and let $\Om$ be a regular $ I$-system of varieties on $V$ over $A$.  For $u \in J$, write $\phi_u = \rho_u\Om(u)$.

Let $\Om_0 = \Om\upharpoonright \partial J$.  For every $u \subseteq V$, write $d_u = dim(\Om_0(u))$ and  $\Pp_u = \mu_{d_u}/(\mu_{d_u}(\Om_0(u)))$, so that $\Pp_u$ is   a probability measure on $\Om(u)$. %Write $\Pp_\emptyset $ for the counting measure.
Then 
\begin{equation}\label{equation: stationarity theorem, 1}
\Pp_V(\bigwedge_{u \in J} \phi_u(\Om_0(V))) = \prod_{u \in J} \Pp_u(\phi_u).
\end{equation}

\end{theorem}

\begin{proof}
Let $\S$ be a system of difference fields on $V$ over $A$ associated to $\Om$ and let $\S_0 = \S \upharpoonright  \partial J$. For $u \in J$, we will view $\S(u)$ as a pure field extending $\S_0(u)$ (in fact we could start with $\S_0$ and consider abstract field extensions isomorphism to $\S$). For every $u \in J$, let $\tau_u$ be an arbitrary extension of $\si$ from $\S_0(u)$ to $\S(u)$ such that $\phi_u$ is satisfied in the structure $(\S(u), \tau_u)$.

%As $\Om$ is adapted to $\Phi$ and each $\phi_u \in \Phi$ is Galois, we note that   $\si$ actually   extends to a difference operator on each $L_u$ and that if we define the $J$-refinement $\S_1$ of $\S$ by setting $\S_1(u) = L_u$, then $\S_1$ is regular. We will not use the extension of $\si$ to $L_u$ but the regularity of $\S_1$ is essential.

 % We want to use Proposition \ref{proposition: alternative characterisation of measure},  and show that we can find an extension $\tau $ of  $\si$ from $\S(E)$ to $L$ such that $(L, \tau) \models \bigwedge_{u \in J} \phi_u(\Om(E), a))$. This will show that $cl_A(\bigwedge_{u \in J} \phi_u(\Om(E))) = \Om(E)$. Counting the possible extensions $\tau$ will then give (\ref{equation: stationarity theorem, 1}).
 
\begin{claim}
The difference operators $\tau_u$    have a common extension $\tau$ to $\S(V)$ and $(\S(V), \tau) \models \bigwedge_{u \in J} \phi_u(\Om_0(V))$
\end{claim}

\begin{proof}[Proof of claim]
 By regularity  and Lemma \ref{lemma: regular systems: third linear disjointness lemma}, we know that $\S(u)$ is linearly disjoint from $\prod_{I} \S(v)^{alg}$ over $  \S_0(u)$. Therefore, we   extend each $\tau_u$  to the field $\S(u) \prod_{I} \S(v)^{alg}$ so that $\tau_u \upharpoonright  \prod_{I} \S(v)^{alg} = \si$. 
 
Fix an enumeration $(u_k)$ of $J$ and suppose we have found a common extension $\tau$ of $\tau_{u_0}, \ldots, \tau_{u_k}$ to $ \prod_{i = 0}^k  \S(u_i) \prod_{v \in I} \S(v)^{alg}$.  By Lemma \ref{lemma: regular systems: third linear disjointness lemma}, $\S(u_{k+1})$ is linearly disjoint from $\prod_{i = 0}^k \S(u_i)^{alg}\prod_{v \in I} \S(v)^{alg}$ over $ \S_0(u_{k+1})$. By the Towers property, $\S(u_{k+1}) \prod_{I} \S(v)^{alg}$ is linearly disjoint from $\prod_{i = 0}^k  \S(u_i) \prod_{I} \S(v)^{alg}$ over $\prod_{I} \S(v)^{alg}$. Moreover, $\tau$ and $\tau_{k+1}$ are both equal to $\si$ on $\prod_{I} \S(v)^{alg}$. This proves that we can find a common extension of the $\tau_u$ to $\prod_{J} \S(v) \prod_{I} \S(v)^{alg}$, and this field contains $\S(V)$.
\end{proof}

 $\bigwedge \phi_u(\Om_0(V))$ is equal to the projection of $\Om(V)$ to $\Om_0(V)$. By Lemma \ref{lemma: measure of system of varieties is measure of product}, $\Om(V)$ has dimension $d_V$, so $\bigwedge \phi_u(\Om_0(V))$ has dimension $d_V$. 
By Proposition \ref{proposition: alternative characterisation of measure} and in the notation of Definition \ref{definition: normalised measure}, it is enough to prove 
\[
\nu(\bigwedge_J \phi_u(\Om_0(V)) = \prod_J \nu(\phi_u).
\]

The claim   proves that  $N(\bigwedge_{u \in J} \phi_u) = \prod_{u \in J} N(\phi_u)$, since our choices of $\tau_u$ were arbitrary. Therefore, we only  need to prove that $[\S(V) : \S_0(V)] = \prod_J [\S(v) : \S_0(v)]$. We have already seen that $\S(u)$ is linearly disjoint from $B:=\prod_{I} \S(v)^{alg}$ over $\S_0(u)$ for every $u \in J$ and that for every $k \geq 0$, $\S(u_{k+1})B$ is linearly disjoint from  $\prod_{i = 0}^k \S(u_i) B$ over $B$, where $(u_k)$ is an enumeration of $J$ as in the proof of the claim. Therefore, working by induction:
\begin{IEEEeqnarray*}{rCl}
[\prod_{i = 0}^{k+1} \S(u_i) B  : B ] & = & [  \prod_{i = 0}^{k+1} \S(u_i) B : \prod_{i = 0}^{k} \S(u_i) B][\prod_{i = 0}^{k} \S(u_i)B:  B] \\
& = & [\S(u_{k+1}) B : B]\prod_{i = 0}^k[\S(u_i)B : B]  =  \prod_{i = 0}^{k+1} [\S(u_i) : \S_0(u_i)]
\end{IEEEeqnarray*}
Therefore $[\prod_{u \in J} \S(u) B  : B ]  = \prod_{u \in J} [\S(u) : \S_0(u)] $. Now 
\[
\prod_{u \in J} [\S(u) : \S_0(u)] = [\prod_{u \in J} \S(u) B  : B ] \leq [\S(V) : \S_0(V)] \leq \prod_{u \in J} [\S(u) : \S_0(u)]
\]
so we deduce $[ \S(V): \S_0(V) ]  = \prod_{u \in J} [\S(u) : \S_0(u)]$. 
 \end{proof}
 
\noindent \textbf{Remarks:} (1) The statement of the stochastic independence theorem  may appear technical at first but we stress that any collection of definable sets $(\phi_u)_{u \in J}$ with  variables suitably indexed by $V$ can always be lifted to a $\partial J$-system $\Om_0$ so that the setting of the stochastic independence theorem holds.

(2) The stochastic independence theorem can be seen as an analogue of the counting lemma usually associated to the Szemer\'edi counting lemma. See \cite{Gowers2006} 3.4 or 4.4 for example.

\medskip

We now prove the stationarity theorem. This is a higher-order and a quantitative version of the stationarity theorem used in \cite{Tao2012}. We show that this theorem follows from the Stochastic Independence Theorem and from additional facts about regular systems of varieties. 

We will not use this theorem in our proof of the algebraic hypergraph regularity lemmas to follow. However, it is striking that it always holds in regular systems and the theorem is of wider interest for model theory and algebraic geometry.

 \begin{corollary}[The Stationarity Theorem]\label{corollary: stationarity theorem}
Let $J$ be an antichain of $V$ with $\bigcup J  = V$. 
%Let $I$ be the simplicial complex on $V$ generated by $\partial J$ and let $\Om$ be a regular $J$-system of varieties on $V$ over $A$.  For $u \in J$ write $\phi_u = \rho_u \Om(u)$. 
Let $I$ be the simplicial complex on $V$ generated by $ J$ and let $\Om$ be a regular $I$-system of varieties on $V$ over $A$.  For $u \in J$ write $\phi_u = \rho_u \Om(u)$. 

Let $\Om_0 = \Om \upharpoonright \partial J$. For every $u \subseteq V$, write $d_u = dim(\Om_0(u))$ and $\Pp_u = \mu_{d_u}/(\mu_{d_u}(\Om_0(u)))$.
%
%and let $\Om$ be an irreducible $I$-system of varieties. Write $d_u = dim(\Om(u))$ for all $u \subseteq V$. Let $\Phi = \{\phi_u \mid u \in J\}$ be a set of Galois formulas over $A$ where $\phi_u \subseteq \Om(u)$. Suppose that $\Om$ is adapted to $\Phi$.
%
% For every $u \subseteq V$, write $\mu_u = \mu_{d_u}$ and $\Pp_u = \mu_{u}/(\mu_{u}(\Om(u)))$.

Fix any nonempty $u \subseteq V$ and a generic point $a_u $ in $\Om_0(u)$. For any $v \subseteq u$, write $a_v$ for the projection of $a_u$ to $\Om_0(v)$. Then we have  
\begin{equation}\label{equation: stationarity theorem, 2}
 \Pp_{V \setminus u} (\bigwedge_{v \in J} \phi_v(\Om_0(V), a_u)) = \prod_{v \in J} \Pp_{v \setminus u}(\phi_v(\Om_0(v), a_{u\cap v})).
\end{equation}
 \end{corollary}

\begin{proof}
  Fix $a_u$ generic in $\Om_0(u)$. Write $J' = \{v \in J \mid v \subseteq u\}$. We can assume that for any $v \in J'$, $a_v$ satisfies $\phi_v$ (otherwise both sides of (\ref{equation: stationarity theorem, 2}) are $0$). In this case,  (\ref{equation: stationarity theorem, 2}) becomes 
\[
 \Pp_{V \setminus u} (\bigwedge_{ J \setminus J'} \phi_v(\Om_0(V), a_u)) = \prod_{J \setminus J'} \Pp_{v \setminus u}(\phi_v(\Om_0(v), a_{u\cap v}))).
\]
 
  For every $v \in J \setminus J'$ and generic $b \in \Om_0(u \cap v)$, the measure of $\phi_{v }(\Om_0(u \cap v), b)$ is constant. This follows from Proposition \ref{proposition: alternative characterisation of measure} and regularity. Therefore, we can apply Fubini to get 
\begin{IEEEeqnarray*}{rCl}
\mu_v( \phi_v(\Om_0(V)) & = & \int_{b \in \Om_0(v \cap u)} \mu_{v \setminus u}(\phi_v(\Om(V, b)) d\mu_{v \cap u}\\
& = & \mu_{u \cap v}(\Om_0(u\cap v)) \mu_{v \setminus u}(\phi_v(\Om_0(v), a_{u\cap v})).
\end{IEEEeqnarray*}
By Lemma \ref{lemma: measure of system of varieties is measure of product}, $\mu_V(\Om_0(V)) = \mu_u(\Om_0(u))\mu_{V \setminus u}(\Om_0(V \setminus u))$. Therefore 
\begin{IEEEeqnarray*}{rCl}
\Pp_v(\phi_v(\Om_0(v))) & = & \frac{\mu_{u  \cap v}(\Om_0(u\cap v)) \mu_{v \setminus u}(\Om_0(v \setminus u))}{\mu_v (\Om_0(v)) } \Pp_{v \setminus u} (\phi_v(\Om_0(v), a_{v \cap u}))\\
& = & \Pp_{v \setminus u} (\phi_v(\Om_0(v), a_{u \cap v})).
\end{IEEEeqnarray*}

Similarly, for generic $b \in \Om_0(u)$, $\mu_{V \setminus u} (\bigwedge_{v \in J \setminus J'} \phi_v(\Om_0(V), b))$ is constant. Therefore, for generic $b \in \bigwedge_{ v \in J'} \phi_v(\Om_0(u))$, $\mu_{V \setminus u} (\bigwedge_{J  } \phi_v(\Om_0(V), b))$ is constant and we deduce by Fubini 
\begin{IEEEeqnarray*}{rCl}
\mu_V(\bigwedge_{v \in J}\phi_v(\Om_0(V))) & = & \mu_u( \bigwedge_{v \in J'}\phi_v(\Om_0(u)) \mu_{V \setminus u}(\bigwedge_{v \in J} \phi_v(\Om_0(V), a_u))
\end{IEEEeqnarray*} 
By Lemma \ref{lemma: measure of system of varieties is measure of product}, we have 
\[
\Pp_V(\bigwedge_{ J} \phi_v(\Om_0(V))) = \Pp_u(\bigwedge_{J'}\phi_v(\Om_0(u)))\Pp_{V \setminus u} (\bigwedge_{ J \setminus J'  } \phi_v(\Om_0(V), a_u)).
\]
By the stochastic independence theorem applied to $\Phi$, we conclude:
\begin{IEEEeqnarray*}{rCl}
\Pp_{V \setminus u} (\bigwedge_{J \setminus J'} \phi_v(\Om_0(V), a_u)) & = & \prod_{J \setminus J'}\Pp_v(\phi_v(\Om_0(v))) \\
& = & \prod_{J \setminus J'} \Pp_{v \setminus u}(\phi_v(\Om_0(v), a_{u\cap v})).
\end{IEEEeqnarray*}
 \end{proof}

\subsection{Quasirandomness}

\emph{As before $K$ is a  model of ACFA   and $A \leq K$ is a perfect inversive subfield. $V$ is a finite set.}

\medskip

 In this section, we show that quasirandomness in the sense of \cite{Gowers2006} follows easily from the stochastic independence  theorem. 
\medskip

Let $\Om$ be a  system of   varieties on $V$ over $A$. Definition \ref{definition: doubling} introduces   the doubling of $\Om$, written $D(\Om)$. In the case where $\Om$ is the Cartesian product $\prod_{i \in V} \Om(i)$,  $D(\Om)$ is obtained by taking copies $\Om(i, 0)$ and $\Om(i, 1)$ of each $\Om(i)$. $D(\Om)$ is the hypergraph on $V \times \{0, 1\}$ consisting of those tuples $(x_{i, j})_{i \in V, j = 0, 1}$ such that for every choice function $\iota : V \to \{0, 1\}$, the tuple $(x_{i , \iota(i)})$ is in $\Om(V)$. When $\Om$ is a system of varieties, $D(\Om)$ takes the form of a fibre product construction.
 Definition  \ref{definition: doubling} also  defines $D(\Om, u)$, the  doubling of $\Om$ outside a designated subset $u$  of $V$. This will be useful in the next section.

In Definition \ref{definition: doubling}, if $u$ is a set, we write $2^u$ for the set of functions $u \to [2]$ where $[2] = \{0, 1\}$.

\begin{definition}\label{definition: doubling}
Let $\Om$ be an $I$-system of   varieties on $V$ over $A$. 

\begin{enumerate}
\item Let $u \subseteq V$ and $\iota \in 2^u$. We identify the pair $(u, \iota)$ with the subset of $V \times [2]$ given by $\{(i, \iota(i)) \mid i\in u\}$.
\item Define $D(I)$, the doubling of $I$, to be the collection of subsets of $V \times [2]$ of the form $(u, \iota)$ where $u \in I$ and $\iota \in 2^u$. We view $I$ as a partial order with the obvious inclusions. Note that $D(I)$ is a simplicial complex on $V \times [2]$.
\item For every $(u, \iota) \in D(I)$, define $\Om(u, \iota)$ to be a copy of the variety $\Om(u)$. When $(v, \iota') \subseteq (u , \iota)$, we have a natural projection $\Om(u, \iota) \to \Om(v, \iota')$
\item Define $D(\Om)$, the doubling of $\Om$, to be the $D(I)$-system of varieties on $V \times [2]$ taking $(u, \iota)$ to $\Om(u, \iota)$. 
\item To make notation lighter, we usually write $D(\Om)$ for the variety in $D(\Om)$ associated to $V \times [2]$ (instead of writing $D(\Om)(V \times [2])$). 
\end{enumerate}

\end{definition}

\noindent \textbf{Remark:} Let $\Om$ be a regular   $I$-system of   varieties on $V$ over $A$.   Lemma  \ref{lemma: measure of system of varieties is measure of product} generalises easily to prove that $D(\Om)$ is  irreducible and hence $D(\Om)$ is regular.

\medskip

If $\phi $ is a definable set contained in $\Om(u)$ and $\iota \in 2^{u}$, write $\phi(\Om(u, \iota))$ for the definable set contained in $\Om(u, \iota)$ obtained by carrying over $\phi$ to $\Om(u, \iota)$. As a set in $K$, $\phi(\Om(u, \iota))$ is equal to $ \phi(\Om(u))$ but it is useful to distinguish them in notation. If we were to rephrase this in terms of formulas with free variables, $\phi(\Om(u, \iota))$ would be the result of substituting the free variables of $\Om(u, \iota)$ for the free variables of $\Om(u)$ in $\phi $. 

We make here some   comments which clarify the technical background of Definition \ref{definition: quasirandom}. In that definition, we consider a system of varieties $\Om$ with $dim(\Om(V))= d$, the definable set $\phi = \rho_V\Om(V)$ and the function $f$  on $\Om(V)^-$  defined by $ \I_{\phi} - \frac{\mu_d(\phi)}{\mu_d(\Om(V)^-)} \I_{\Om(V)^-}$. $f$ is definable, in the sense that it can be viewed as a   function from a type space over $K$ in the appropriate variables to the interval $[-1, 1]$ which is continuous with respect to the logic topology.   Since the measure $\mu_d$ extends to a Borel measure on the types over $K$ contained in $\Om(V)^-$, we can integrate $f$ on $\Om(V)^-$. When $K$ is $\omega_1$-saturated,   it is possible to move away from the type space and integrate $f$ directly over the set $\Om(V)^-$ in $K$, with respect to the $\si$-algebra of $K$-definable sets.
In any case,  the identity $\int_{D(\Om(V)^-)} f d\mu_d = 0$ expresses  an identity which does not depend on any model. 

When $f$ is a definable function on $\Om(V)^-$ and $\iota \in 2^V$, we will write $f^\iota$ for the result of carrying $f$ over to $\Om(V, \iota)^-$. Thus, for $f$ defined above, $f^\iota = \I_{\phi(V, \iota)} - \frac{\mu_d(\phi)}{\mu_d(\Om(V)^-)} \I_{\Om(V, \iota)^-}$. We also write $f^\iota$ for the result of pullback back from $\Om(V, \iota)$ to $D(\Om)$ when the context is clear.  Therefore we can make sense of the integral $\int_{D(\Om)} \prod_{\iota \in 2^V} f^\iota d\mu_{2d}$.

\medskip
Definition \ref{definition: quasirandom} follows Definition 6.3 in \cite{Gowers2006}.

\begin{definition}\label{definition: quasirandom}
Let $\Om$ be a system of   varieties on $V$ over $A$. We say that $\Om$ is quasirandom if the following holds:

For every $ u \subseteq V$ with $|u| \geq 2$, write $d_u = dim(\Om(u))$, let $\phi_u = \rho_u\Om(u)$ and define the function $f_u : \Om(u)^- \to [-1, 1]$ by $f_u = \I_{\phi_u} - \frac{\mu_{d_u}(\phi_u)}{\mu_{d_u}(\Om(u)^-} \I_{\Om(u)^-}$.  Then 
\[
\int_{D(\Om \upharpoonright P(u)^-)} \prod_{\iota \in 2^u} f_u^\iota d\mu_{2d_u} = 0.
\]
\end{definition}

 The following proposition is essential for the next section, but it follows easily from the stochastic independence theorem. 
  
\begin{proposition}\label{proposition: quasirandomness}
Let $\Om$ be a regular system of varieties on $V$ over $A$. Then $\Om$ is quasirandom.
\end{proposition}

\begin{proof}
 We check quasirandomness at level $V$. 
Suppose $dim(\Om(V)) = d$.   Write $\Pp_d = \mu_d/(\mu_d(\Om(V))$ and $\Pp_{2d}  = \mu_{2d}/(\mu_d(\Om(V))^2$. Let $\phi = \rho_V\Om(V)$.

Let $f  =\I_{\phi} - \frac{\mu_d(\phi )}{\mu_d(\Om(V)^-)}\I_{\Om(V)^-}$ on $\Om(V)^-$. By the stochastic independence theorem, the set of functions $\{f^\iota \mid \iota \in 2^V\}$ on $D(\Om \upharpoonright P(V)^-)$ is   independent in the probabilistic sense. Therefore  
\[
\int_{D(\Om \upharpoonright P(V)^-)} \prod_{\iota \in 2^V} f^\iota d\Pp_{2d} = \prod_{\iota \in 2^V} \int_{\Om(V, \iota)^-} f^\iota d\Pp_d = 0.
\]
\end{proof}

In Section \ref{subsection: equivalent notions}, we will see that if $\Om$ is quasirandom, then $\Om$ is definably \'etale-edge-uniform, and hence $\Om$ is regular. This proof goes via the fundamental equivalence of Gowers discussed in the next section. It would be interesting to find a proof of this result at the algebraic level.

 \section{A combinatorial approach to algebraic hypergraph regularity}  \label{section: combinatorics}
  
  In this section, we leave models of ACFA behind and we  apply our theorems to derive combinatorial results concerning finite definable sets in the difference fields $K_q$ where $q$ is a power of a prime $p$, $K_q$ is the algebraic closure of the field containing $p$ elements and the difference operator on $K_q$ is the $q$-th power $ x \mapsto x^q$.

  \subsection{Combinatorial notions and asymptotics}\label{subsection: combinatorial notions}

\emph{We fix $A$ a perfect inversive finitely generated difference field. $V$ is a finite set.}

\medskip
  We recall the notion of Frobenius specialisation from \cite{HrushovskiFrob}, section 12. Suppose that we have some finite definable data $D$ defined over a finitely generated  perfect inversive difference field $A$  (e.g. $D$ is a variety, a definable set, or a system of varieties).  We will say that some property $P$ holds of $D$ for \emph{almost all $q$} to mean the following:  there is a sufficiently large finitely generated  difference domain $R$ in $A$ such that $D$ is defined over $R$ and for   any sufficiently large prime power $q$ and for any difference ring homomorphism $h : R \to K_q$, $P$ holds of $D^h$ in $K_q$, where  $D^h$ is the definable data in $K_q$ obtained by applying $h$ to the parameters in $D$ and interpreting $\si$ as the $q$-th power. 
  
  Equivalently, if $D$ is some data definable over a difference domain $R$, we say $P$ holds of $D$ for almost all $q$ if there is $c \in R$ such that for any sufficiently large prime power $q$ and any homomorphism $h : R \to K_q$ with $h(c) \neq 0$, $P$ holds of $D^h$ in $K_q$. See \cite{HrushovskiFrob} for an in-depth discussion of Frobenius specialisations.   
   
In an effort to make notation lighter, we always consider data $D$ over an inversive difference field $A$ and we suppress the reference to the difference domain $R$ and to the homomorphism $h$  when discussing specialisations. It will always be clear from context whether we are working with $D$ over $A$ or if we are working with specialisations.

The fundamental theorem about Frobenius specialisations is  the twisted Lang-Weil estimates due to Hrushovski:

\begin{theorem}[\cite{HrushovskiFrob}\footnote{While the manuscript we reference is dated from 2022, the twisted Lang-Weil estimates were discovered by Hrushovski in the early 2000s.}]\label{theorem: twisted lang-weil}
Let $X$ be a variety over $A$ and suppose that $X$ has finite total dimension $d$. Then  for almost all $q$, specialising $X$ to $K_q$, we have
\[
|X(K_q) - \mu_d(X)q^d|  = O(q^{d-1/2}).
\]
where $O(\cdot)$ depends only on the degree of $X$.
\end{theorem}

Theorem \ref{theorem: twisted lang-weil} extends to Galois formulas over $A$ in the obvious way. Since ACFA is the asymptotic theory of the difference fields $K_q$, Theorem \ref{theorem: twisted lang-weil} gives asymptotics for all definable sets of finite total dimension.

  \medskip
  
We now extend the notions of quasirandomness and edge-uniformity introduced in the previous section so that we can apply them inside the structures $K_q$.

\begin{definition}\label{definition: e-quasirandom}
let $\Om$ be a  system of   varieties on $V$ over $A$. For every $u \subseteq V$ with $|u| \geq 2$, write  $\phi_u = \rho_u\Om(u)$ and let $\e_u$ be a function $\Nn \to [0, \infty)$. Write $\e = (\e_u)_{|u| \geq 2}$. 

We say that $\Om$ is $\e$-quasirandom if for almost all $q$ and all $u \subseteq V$ with $|u| \geq 2$, specialising  $\Om$ to $K_q$ and writing $f_u = \I_{\phi_u} - \frac{|\phi_u|}{|\Om(u)^-|} \I_{\Om(u)^-}$, we have 

\[
\sum_{D(\Om\upharpoonright P(u)^-)} \prod_{\iota \in 2^{u}} f_u^\iota = O(\e_u(q))|\Om(u)^-|^2
\] 
where $O(\cdot)$ only depends on the degrees of the varieties in $\Om$. 
 
 \end{definition}

 In the next definition, we introduce chains which are a purely combinatorial notion, since they exist in finite sets inside $K_q$ and we are not interested in identifying them as definable sets. This terminology comes from \cite{Gowers2006}.
 
\begin{definition}\label{definition: chains and edge uniformity}
Let $\Om$ be a system of   varieties on $V$ over $A$. Specialise $\Om$ to some $K_q$. 
 
 \begin{enumerate}
  
  \item    A chain $W = (W(v))_{v\in P(V)}$   in $\Om$ is a collection of sets  such that for all $v \subseteq V$, $W(v) \subseteq \Om(v)$ and for any $v \subseteq u$, $W(v)$ contains the image of $W(u)$ under the projection $\Om(u) \to \Om(v)$.

 \item  Let $W$ be a chain   in $\Om$. Let $I$ be a simplicial complex on $V$. We say that $W$ is an $I$-chain if for all $u \notin I$, $W(u)$ is the fibre product of the sets $(W(v))_{v \in P(u)^-}$. 
 
 \item Let $W$ be a chain in $\Om$. For every $u \subseteq V$, write $W(u)^-$ for the fibre product $\prod(W(v), v \in P(u)^-)$ and $\rho_uW(u)$ for the projection $W(u) \to W(u)^-$.

\item   Let $I$ be a simplicial complex on $V$. An  $I$-chain decomposition of $\Om$ is a collection  $(\mathcal{W}_v)_{v \in I}$ such that each $\mathcal{W}_v$ is a partition of $\Om(v)$ and for every $u \in I$ and every $X \in \mathcal{W}_u$, there is a $P(u)^-$-chain $W$ such that $X \subseteq W(u)$ and for every $v\in P(u)^-$, $W(v) \in \mathcal{W}_v$. 

We say that an $I$-chain $W$ is contained in the $I$-chain decomposition $(\mathcal{W}_v)_{v \in I}$ if for all $v \in I$, $W(v) \in \mathcal{W}_v$. 
\end{enumerate}
\end{definition}

 The next definition makes precise the notion of $\e$-edge-uniformity. We also define $\e$-\'etale edge-uniformity in order to relate  $\e$-edge-uniformity to definable edge-uniformity, but we will  soon see that $\e$-\'etale edge-uniformity is equivalent to $\e$-edge-uniformity. The notion of definable $\e$-\'etale-edge-uniformity will be used only in Corollary \ref{corollary: asymptotic quasirandomness}.

\begin{definition}\label{definition: combinatorial edge uniformity}
Let $\Om$ be a system of   varieties on $V$ over $A$. For every $ u\subseteq V$ with $|u| \geq 2$, let $\phi_u = \rho_u\Om(u)$ and let $\e_u$ be a function $\Nn \to [0, \infty)$. Write $\e = (\e_u)_{|u| \geq 2}$. 

% and let $\phi$ be a   formula over $A$ contained in $\Om(V)$ with $dim(\phi) = dim(\Om(V))$. For every $u \in P(V)$, let $\e_u : \Nn \to [0, \infty)$. Write $\e = (\e_u)_{u \in P(V)}$.
\begin{enumerate}
\item  We say that $\Om$ is $\e$-edge-uniform  if for almost all $q$, specialising $\Om$ and $\phi$ to $K_q$, for all $u \subseteq V$ with $|u| \geq 2$, if $W$ is  $P(u)^-$-chain contained in $\Om \upharpoonright P(u)^-$, then 
\[
|\frac{|\phi_u \cap W(u)|}{|\phi_u|} - \frac{|W(u)|}{|\Om(u)^-|}|  = O(\e_u(q))
\]
where $O(\cdot)$ only depends on the degrees of the varieties in $\Om$.

\item We say that $\Om$ is definably $\e$-\'etale-edge-uniform if for all $ u\subseteq V$ with $|u| \geq 2$, for any $P(u)^-$-refinement $\Om'$ of $\Om \upharpoonright P(u)^-$ over a finitely generated algebraic extension  $A'$ of $A$ with projection $\pi : \Om'(u) \to \Om(u)^-$, for almost all $q$, we have 
\[
|\frac{|\phi_u \cap \pi(\Om'(u))|}{|\phi_u|} - \frac{|\pi(\Om'(u))|}{|\Om(u)^-|}| = O(\e_u(q))
\]
where $O(\cdot)$ only depends on the degrees of the varieties in $\Om$.  

\item We say that $\Om$ is $\e$-\'etale-edge-uniform if for all $ u\subseteq V$ with $|u| \geq 2$, for any $P(u)^-$-refinement $\Om'$ of $\Om \upharpoonright P(u)^-$ over a finitely generated algebraic extension $A'$ of $A$ with projection $\pi : \Om'(u) \to \Om(u)^-$, for almost all $q$ and for any $P(u)^-$-chain $W$ in $\Om'$, we have 
\[
|\frac{|\phi_u \cap \pi(W(u))|}{|\phi_u|} - \frac{|\pi(W(u))|}{|\Om(u)^-|}| = O(\e_u(q))
\]
where $O(\cdot)$ only depends on the degrees of the varieties in $\Om$.  
 
\end{enumerate}
\end{definition}

The following corollary is just a reformulation of Theorem \ref{theorem: definable etale edge uniformity in ACFA} and a direct application of Theorem \ref{theorem: twisted lang-weil}.

\begin{corollary}\label{corollary: asymptotic quasirandomness}

Let $\Om$ be a system of   varieties on $V$ over $A$. For every $u \subseteq V$ with $|u| \geq 2$, set $\e_u(q) = q^{-1/2}$ and let $\e = (\e_u)_{|u| \geq 2}$.

Then there is a surjective   $P(V)^-$-refinement $\Om'$ of $\Om$ over some finitely generated   extension $A'$ of $A$ with regular components $\Om_1, \ldots, \Om_N$ such that for every $i$, $\Om_i$ is definably $\e$-\'etale-edge-uniform and $\e$-quasirandom.
$N$ depends only on the degrees of the varieties in $\Om$.
 \end{corollary}
 \subsection{Equivalences between defined notions}\label{subsection: equivalent notions}
 
\emph{We fix $A$ a perfect inversive finitely generated difference field. $V$ is a finite set.} 
 \medskip

We first show that the notions of $\e$-edge-uniformity and $\e$-\'etale-edge-uniformity coincide. We state the lemma in the language of systems of varieties but this is really a purely combinatorial result. We use the following technical lemma:

\begin{lemma}\label{lemma: sections}
Let $I$ be a simplicial complex on $V$. Let $\Om$ be a system of varieties on $V$ over $A$ and let $\Om'$ be an $I$- refinement of $\Om$ with projection $\pi : \Om' \to \Om$. For every $u \subseteq V$, write $d_u = dim(\Om(u))$. 

Then there is some $N$ such that for almost all $q$, specialising $\Om$ and $\Om'$ to $K_q$, the following holds: there is an $I$-chain decomposition $(\mathcal{W}_v)_{v \in I}$ of $\Om'$ containing at most $N$ chains such that for every $u \in I$ and every set $X \in \mathcal{W}_u$, the map $\pi : X \to \Om(u)$ is injective.
\end{lemma}

\begin{proof}
 
We construct $(\mathcal{W}_v)_I$ inductively. The base case where $I$ is the set of singletons of $V$ is clear. Suppose that $I = I' \cup \{u\}$ where $P(u)^- \subseteq  I'$ and $(\mathcal{W}_v)_{I'}$ have been constructed. 

Writing $W_1, \ldots, W_n$ for the $I'$-chains contained in $(\mathcal{W}_v)_{I'}$, we find that for every $i \leq n$, the projection $W_i(u) \to \Om(u)^-$ is injective and that the sets $W_1(u), \ldots, W_n(u)$ partition $\Om'(u)^-$. Now take $\mathcal{W}_u$ to be the partition of $\Om'(u)$ obtained by taking sections of each projection $\rho_u : \rho_u^{-1}(W_i(u)) \to W_i(u)$. The size of $\mathcal{W}(u)$ depends only on $n$ and the multiplicity of the projection $\rho_u $. 

Take $X \in \mathcal{W}_u$. To see that $\pi : X \to \Om(u)$ is injective, note that the maps $\Om'(u) \to \Om(u) \to \Om(u)^-$ and $\Om'(u) \to \Om'(u)^- \to \Om(u)^-$ form a commutative diagram. We know that $X \to  \Om'(u)^- \to \Om(u)^-$ is injective, so the result follows.
\end{proof}

\begin{lemma}\label{lemma: edge uniform same as etale-edge-uniform}
Let $\Om$ be a  system of difference varieties on $V$ over $A$.  
For $u \subseteq V$ with $|u| \geq 2$,  let $\e_u : \Nn \to [0, \infty)$ and write $\e = (\e_u)$. 
Then $\Om$ is $\e$-\'etale-edge-uniform  if and only if $\Om$ is $\e$-edge-uniform.
\end{lemma} 

\begin{proof}
Let $\phi = \rho_V \Om(V)$ and let $\Om'$ be a $P(V)^-$-refinement of $\Om \upharpoonright P(V)^-$. Write $\pi$ for the projection $\Om' \to \Om$. It is enough to show that if  $\Om$ is $\e$-edge-uniform, then, specialising to some $K_q$, for every $P(V)^-$-chain $W$ in $\Om'$, 
\[
|\frac{|\phi \cap \pi(W(V))|}{|\phi |} - \frac{|\pi(W(V))|}{|\Om(V)|}| = O(\e_V(q)).
\]
 
Fix $W$ a chain in $\Om'$. Let $(\mathcal{W}_u)_{P(V)^-}$ be a chain decomposition of $\Om'$ given by Lemma \ref{lemma: sections}. Let $(W_i)$ be an enumeration of the chains in $(\mathcal{W}_u)_{P(V)^-}$. Define $W_i' = W \cap W_i$. 

For every $i$, $\pi(W_i'(V))$ is   equal to the fibre product $\prod(\pi( W_i'(v)), v \in P(V)^-)$, since $\pi$ is  injective on each $W_i'(v)$.  Write $Z_i(u) =  \pi(W_i'(u))$ so that each $Z_i$ is a $P(V)^-$-chain  in $\Om$. It follows that $\pi(W(V)) = \pi(\bigcup_i W_i'(V)) = \bigcup_i Z_i(V)$. Now $\e$-edge-uniformity applied to each chain $Z_i$ gives the result.
\end{proof}

\noindent \textbf{Remark:} Back in definition \ref{definition: combinatorial edge uniformity}, we could have defined  the notion of \emph{definable $\e$-edge-uniformity} in the expected way, but we chose not to discuss this notion. The reason for this is that  Lemma \ref{lemma: edge uniform same as etale-edge-uniform} does not allow us to show that definable $\e$-edge-uniformity is   equivalent to definable $\e$-\'etale-edge-uniformity. This is because the proof of Lemma \ref{lemma: edge uniform same as etale-edge-uniform} goes via the sections constructed in Lemma \ref{lemma: sections}, which are not definable.  It would be interesting to know if Lemma \ref{lemma: edge uniform same as etale-edge-uniform} can be extended to the definable setting.

  \medskip

Combining Proposition \ref{proposition: definable edge uniformity and regular systems} and Proposition  \ref{proposition: quasirandomness}, we find that we have already proved that  if $\phi$ is definably $q^{-1/2}$-\'etale-edge-uniform with respect to $\Om$ then $\phi$ is $q^{-1/2}$-quasirandom with respect to $\Om$. Suitably transposed to our setting, this is close to one direction of the fundamental Theorem 4.1 in \cite{Gowers2006}. Proposition \ref{proposition: edge uniformity to quasirandom} proves this in full generality, with general error bounds. We follow the proof of \cite{Gowers2006}. This result  will be needed for Theorem \ref{theorem: combinatorial regularity, classical partitions}.

\begin{proposition}\label{proposition: edge uniformity to quasirandom}
Let $\Om$ be a    system of   varieties on $V$ over $A$. 
For every $u \subseteq V$ with $|u| \geq 2$, let $\e_u : \Nn \to [0, \infty)$. Write $\e = (\e_u)_{u \in P(V)}$. 
If $\Om$ is   $\e$-\'etale-edge-uniform   then $\Om$ is $\e$-quasirandom.
 \end{proposition}

\begin{proof}
It is enough to check quasirandomness at the top level $V$. Let $\phi = \rho_V \Om(V)$ and write $\Om_0 = \Om \upharpoonright P(V)^-$ (so that $\Om_0(V) = \Om(V)^-$). We specialise $\Om$  to some $K_q$ and write $f = \I_{\phi} - \frac{|\phi|}{|\Om(V)^-|} \I_{\Om(V)^-}$.

   Write $\iota_0, \iota_1 \in 2^{V}$ for the constant functions equal to $0$ and $1$ respectively. Then $D(\Om_0)$ is a finite cover of the  Cartesian product of $\Om_0(V, \iota_0)$ and $\Om_0(V, \iota_1)$. 
    For any $x \in \Om_0(V, \iota_0)$ and $y \in \Om_0(V, \iota_1)$, we write $D(\Om_0)(xy)$ for the finite set of elements $z \in D(\Om_0)$ which project onto $x$ and onto $y$. The cardinality of $D(\Om_0)(xy)$ is bounded across $q$. By the triangle inequality we obtain
\begin{equation}\label{equation: edge-uniform to quasi-random, 1}
|\sum_{D(\Om_0)} \prod_{\iota \in 2^V} f^\iota | \leq  \sum_{y \in \Om_0(V, \iota_1)}\Big( |f^{\iota_1}(y)| \cdot |\sum_{x \in \Om_0(V, \iota_0)} f^{\iota_0}(x) \sum_{D(\Om_0)(xy)} \prod_{\iota \neq \iota_0, \iota_1} f^\iota|\Big).
\end{equation}
   Fix $b  \in \Om_0(V, \iota_1)$, and for every $u \in P(V)^-$, write $b_u$ for the projection of $b$ to $\Om(u, \iota_1 \upharpoonright u)$. For every $u \in P(V)^-$, define the function $\iota_u \in 2^V$ by $\iota_u \upharpoonright u = 0$ and $\iota_u \upharpoonright u^c = 1$.
 
  We  define   a $P(V)^-$-refinement $\Om_b$ of $\Om_0(V, \iota_0)$ as follows:  for every $u \in P(V)^-$ and every $a \in \Om_0(u, \iota_0 \upharpoonright u)$, the pullback of $a$ to $\Om_b(u)$ is the set of elements $c$ such that there is $v \subseteq u$ with $c \in \Om_0(V, \iota_v)$ and the projection of $c$ to $v$ equals $a_v$ and the projection of $c$ to $v^c$ equals $b_{v^c}$. It is clear how to set up projections $\Om_b(u) \to \Om_b(v)$ for $v \subseteq u$. Write $\pi$ for the projections $\Om_b \to \Om_0$.\footnote{\label{footnote: edge-uniform to quasirandom} Note that $\Om_b$ is   defined over $Ab$ as a system of difference varieties, but we could equally well define a refinement of $\Om$ over $A$ with multiplicities matching those of $\Om_b$ and set up identification maps with $\Om_b$. This would amount to exactly the same, from a combinatorial point of view.  }

Observe that $\sum_{D(\Om_0)(xb)} \prod_{\iota \neq \iota_0, \iota_1} f^\iota$ is a function of $x \in \Om_0(V, \iota_0)$ which takes finitely many values, according to the number of elements above $x$ in $\Om_b(V)$ which belong to $\phi^\iota$ or not. Therefore, there is some $N$ (which does not depend on $q$), some scalars $\lambda_1, \ldots, \lambda_N$, and some chains $W_1, \ldots, W_N$ contained in $\Om_b$ such that 
\[
\sum_{D(\Om_0)(xb)} \prod_{\iota \neq \iota_0, \iota_1} f^\iota = \sum_{j \leq N} \lambda_j \I_{\pi(W_j(V))}
\]
 By $\e$-\'etale-edge-uniformity, we have 
 \begin{equation}\label{equation: edge-uniform to quasirandom, 2}
|\sum_{x \in \Om_0(V, \iota_0)} f^{\iota_0}(x) \sum_{D(\Om_0)(xb)} \prod_{\iota \neq \iota_0, \iota_1} f^\iota | = O(\e_V(q))|\Om(V)|
 \end{equation}
Combining (\ref{equation: edge-uniform to quasi-random, 1}) and (\ref{equation: edge-uniform to quasirandom, 2}), we have 
\[
|\sum_{D(\Om_0)} \prod_{\iota \in 2^V} f^\iota |  = O(\e_V(q))|\Om(V)|^2.
\]
 
\end{proof}

%\noindent \textbf{Remark:} We remarked in the proof of Proposition \ref{proposition: edge uniformity to quasirandom} that the refinement $\Om_b$ does not fit strictly the conditions of Definition \ref{definition: combinatorial edge uniformity} because $\Om_b$ is not defined as a system of difference varieties over $a$. We remarked that this is not an issue since Proposition \ref{proposition: edge uniformity to quasirandom} adopts a combinatorial point of view and $\Om_b$ can be identified with an $A$-definable refinement of $\Om$. 
%
%In fact, we note that $\Om_b$ is an $Ab$-definable system of difference varieties, and that the proof of Proposition \ref{proposition: edge uniformity to quasirandom} only requires us to consider sufficiently generic $b \in \Om(V, \iota_1)$. Therefore, working in a large model $K$ of ACFA containing $A$ and using an internalisation argument in the style of Lemma \ref{lemma: linear disjointness in systems}, we can adapt the proof of Proposition \ref{proposition: edge uniformity to quasirandom} to give a different proof of Theorem \ref{theorem: quasirandomness}.  

We now prove the near-converse of Proposition \ref{proposition: edge uniformity to quasirandom}. The proof essentially follows   \cite{Gowers2006} but requires an application  of regularity of the system of varieties.

 \begin{proposition}\label{proposition: quasirandomness to edge-uniformity}
Let $\Om$ be a regular system of difference varieties on $V$ over $A$.  For every $u \subseteq V$ with $|u| \geq 2$, let $\e_u : \Nn \to [0, \infty)$ and let $\e = (\e_u)_{|u| \geq 2}$. 
 If  $\Om$ is $\e$-quasirandom, then $\Om$ is $(\e_u^{2^{-|u|}})$-edge-uniform.  
 \end{proposition}
 
\begin{proof}
It is enough to check edge-uniformity at the top level $V$.  We specialise all data to some $K_q$. 
Let $\phi = \rho_V \Om(V)$ and let $W$ be a $P(V)^-$-chain contained in $\Om \upharpoonright P(V)^-$. Write $\Om_0 = \Om \upharpoonright P(V)^-$ so that $\phi \subseteq \Om_0(V)$.  In the next expression we write $\I_{W(v)}$ for the indicator function of the set of elements of $\Om_0(V)$ which project to $W(v)$ under the map $\Om_0(V) \to \Om(v)$.
    Define 
    \[
    \Delta =  ||\phi \cap W(V)| - \frac{|\phi|}{|\Om_0(V)|} |W(V)| |= |\sum_{\Om_0(V)}f \prod_{v \in P(V)^-} \I_{W(v)}|.
    \]
 We want to show that when $|V| = n$, 
\[
\Delta^{2^n}  = O(|\Om_0(V)|^{2^{n}-2}) |\sum_{D(\Om_0)} \prod_{\iota \in 2^V} f^{\iota}|.
\]
We proceed by induction on $n$. The induction holds for an arbitrary function $f$ but requires the underlying system to be the specialisation of a regular system. The base case for the induction is $n= 2$, which is exactly Theorem 3.1 in \cite{Gowers2006}, since $P(2)^-$-systems are just Cartesian products.

We will use the following notation: for $v \subseteq u \subseteq V$ and $x \in \Om(v)$, write $\Om(u)(x)$ for the pullback of $x$ to $\Om(u)$ by the map $\Om(u) \to \Om(v)$. 

Let $*$ be an element of $V$ and write $V' = V \setminus \{*\}$ and $V^+ = V' \cup \{(*, 0), (*, 1)\}$. Let $\Om_1$ be the system of varieties on $V^+$ obtained from $\Om_0$ by ``doubling $*$'': $\Om_1((*, 0))$ and $\Om_1((*, 1))$ are two copies of $\Om_0(*)$ and restricting $\Om_1$ to $V^+ \setminus \{(*, 0)\}$ or $V^+ \setminus \{(*, 1)\}$ gives isomorphic copies of $\Om_0(V)$ which project to $\Om_0(V')$. For $v \subseteq V'$, we write $v0 = v \cup \{(*, 0)\}$, $v1 = v \cup \{(*, 1)\}$ and $v01 = v \cup \{(*, 0), (*, 1)\}$. For $i = 0,1$ we write $W(vi)$ for the copy of $W(v \cup \{*\})$ inside $\Om_1(vi)$.

By Cauchy-Schwartz and then expanding, we have 
\begin{IEEEeqnarray}{rCl}
\Delta^{2^n}   & = & \Big(\sum_{x \in \Om_0(V')}  \I_{W(V')}(x)  \sum_{\Om_0(V)(x)} f \prod_{\substack{v \in P(V)^- \\ v \neq V'}} \I_{W(v)}  \Big)^{2^n} \IEEEnonumber \\
& \leq & \Big( \sum_{x\in \Om_0(V')}\I_{W(V')}(x)  \Big)^{2^{n-1}} \Big( \sum_{x \in \Om_0(V')} \big(\sum_{\Om_0(V)(x)} f \prod_{\substack{v \in P(V)^- \\ v \neq V' }} \I_{W(v)}  \big)^2 \Big)^{2^{n-1}} \IEEEnonumber\\
& \leq & |\Om_0(V')|^{2^{n-1}} \Big( \sum_{\Om_1(V^+)} g   \prod_{\substack{v \in P(V')^- }}  \I_{W(v0)}\I_{W(v1)}  \Big)^{2^{n-1}}\label{equation: quasirandom to edge uniform, 1}
\end{IEEEeqnarray}
where $g$ is defined to be the product of the two copies of $f$ inside $\Om_1(V^+)$.

Now we decompose $\Om_1(V^+)$ along the two copies of $\Om_0(*)$. For any $(x, y) \in (\Om_0(*))^2$, write $\Lambda_{xy}$ for the pullback of $(x, y)$ to $\Om_1(V^+)$.   There is a finite projection $\Lambda_{xy} \to \Om_0(V')$ so $|\Lambda_{xy}| = O(|\Om_0(V')|)$ where the constant in $O$ does not depend on $x$ or $ y$.    

%For $u \subseteq V'$, write $d(u) \in D(P(V),V')$ as $(u \times \{0\}) \cup \{(*, 0), (*, 1)\}$. 
$\Lambda_{xy}$ is  a $P(V')$-system when we define $\Lambda_{xy}(u)$  for $u \subseteq V'$ as the pullback of $(x, y)$ to   $\Om_1(u01)$. For every $v \subseteq V'$, define $W(v01) \subseteq \Lambda_{xy}(v)$ to be the fibre product of the copies of $W(v0)$ and $W(v1)$. 

Let $\Gamma_{xy}$ be the restriction of $\Lambda_{xy}$ to $P(V')^-$. Observe that $\Gamma_{xy}$ ``forgets'' $\Om_0(V')$ and hence there is a finite projection $\Lambda_{xy}(V') \to \Gamma_{xy}(V')$ corresponding to the projection $\Om_0(V') \to \Om_0(V')^-$.  For $a \in \Gamma_{xy}(V')$, write  $\Lambda_{xy}(V')(a)$ for the pullback of $a$ to $\Lambda_{xy}(V')$ and define
%\[
%h_{xy}(a) = \sum_{b \in \Lambda_{xy}(V')(a)}g(b) \I_{\Lambda_{xy}(d(V'))}.
%\]
\[
h_{xy}(a) = \sum_{b \in \Lambda_{xy}(V')(a)}g(b).
\]
Fixing   $(x,y) \in \Om\{*\} \times \Om\{*\}$ we have 

\begin{IEEEeqnarray*}{rCl}
 \sum_{b \in \Lambda_{xy}(V') } g(b)   \prod_{\substack{v \in P(V')^-  }}\I_{W(v0)} \I_{W(v1)}(b)  
  & = & \sum_{a \in \Gamma_{xy}(V')} h_{xy}(a) \prod_{\substack{v \in P(V')^-  }}   \I_{W(v01)} (a) \IEEEyesnumber \label{equation: quasirandom to edge uniform, 2}
\end{IEEEeqnarray*}
Combining   (\ref{equation: quasirandom to edge uniform, 1}) and (\ref{equation: quasirandom to edge uniform, 2}) and applying Jensen's inequality, we have:
\begin{IEEEeqnarray*}{rCl}
\Delta^{2^n}   & \leq  & |\Om(V')|^{2^{n-1}} \Big( \sum_{(x, y) \in (\Om(*))^2}  \sum_{\Gamma_{xy}(V')} h_{xy} \prod_{v \in P(V')^-} \I_{W(v01)} \Big)^{2^{n-1}} \\
& \leq & |\Om(V')|^{2^{n-1}}|\Om(*)|^{2^n-2}  \sum_{(x, y) \in (\Om(*))^2} \Big(\sum_{\Gamma_{xy}(V')} h_{xy} \prod_{v \in P(V')^-} \I_{W(v01)} \Big)^{2^{n-1}} \IEEEyesnumber \label{equation: quasirandom to edge uniform, 3}
\end{IEEEeqnarray*}
By induction hypothesis, we have for every sufficiently generic pair $x, y$ in $\Om(*)$ 
\begin{IEEEeqnarray*}{rCl}
\Big(\sum_{\Gamma_{xy}(V')} h_{xy} \prod_{v\in P(V')^-} \I_{W(v01)} \Big)^{2^{n-1}} & = &  O(|\Gamma_{xy}(V')|^{2^{n-1}-2}) \sum_{D(\Gamma_{xy})} \prod_{\iota \in 2^{V'}} h_{xy}^\iota\IEEEyesnumber \label{equation: quasirandom to edge uniform, 4}
\end{IEEEeqnarray*}

\begin{claim}
For any $(x, y) \in (\Om\{*\})^2$, $\sum_{D(\Gamma_{xy})} \prod_{\iota \in 2^{V'}} h_{xy}^\iota = \sum_{D(\Lambda_{xy})} \prod_{\iota \in 2^{V'}} g^{\iota}$. 
\end{claim}

\begin{proof}[Proof of Claim]
This proof is essentially a definition chase. For $c \in D(\Gamma_{xy})$ and $\iota \in 2^{V'}$, write $c_\iota$ for the projection of $c$ to $\Gamma_{xy}(V',\iota)$.  By definition, 
\[
\sum_{c \in D(\Gamma_{xy})} \prod_{\iota \in 2^{V'}} h_{xy}^\iota(c) =  \sum_{c \in D(\Gamma_{xy})} \prod_{\iota \in 2^{V'}} \sum_{b \in \Lambda_{xy}(V', \iota)(c_\iota)}g^\iota(b) 
\]
where $\Lambda_{xy}(V', \iota)$ is a copy of $\Lambda_{xy}$ inside $D(\Lambda_{xy})$. Fix $c \in D(\Gamma_{xy})$. Let $\Xi$ be the set of choice functions $2^{V'} \to \bigsqcup_{\iota \in 2^{V'}} \Lambda_{xy}(V', \iota)(c_{\iota})$. 
Expanding, we have 
\[
 \prod_{\iota \in 2^{V'}} h_{xy}^\iota(c) = \sum_{\xi \in \Xi} \prod_{\iota \in 2^{V'}} g^{\iota}(\xi(\iota)).
\]
Since $D(\Lambda_{xy})$ is the fibre product of the varieties $\Lambda_{xy}(V', \iota)$, any $d \in D(\Lambda_{xy})$ is uniquely determined by its projections to each $\Lambda_{xy}(V', \iota)$. Since $D(\Lambda_{xy})$ is a $D(P(V'))$-system, each set $(V', \iota) \in D(P(V'))$ is maximal and hence we can define a bijection $\Xi \to D(\Gamma_{xy})(c)$ such that $\xi \mapsto (d_\iota)_{\iota \in 2^{V'}}$. This implies that 
\[
 \prod_{\iota \in 2^{V'}} h_{xy}^\iota(c) = \sum_{d \in D(\Gamma_{xy})(c)} \prod_{\iota \in 2^{V'}} g^{\iota}(d_\iota).
\]
The claim follows.
 \end{proof}
By the claim and (\ref{equation: quasirandom to edge uniform, 4}), we have 
\begin{equation}\label{equation: quasirandom to edge uniform, 5}
\Big(\sum_{\Gamma_{xy}} h_{xy} \prod_{v\in P(V')^-} \I_{W(v01)} \Big)^{2^{n-1}} = O(|\Lambda_{xy}(V')|^{2^{n-1}-2}) \sum_{D(\Lambda_{xy})} \prod_{\iota \in 2^{V'}} g^{\iota}.  
\end{equation}
 
By Lemma \ref{lemma: measures of varieties in cover well behaved}, Theorem \ref{theorem: twisted lang-weil}, and Lemma \ref{lemma: explicit definable measure on varieties}, $|\Lambda_{xy}(V')| = O(|\Om(V')|)$ for sufficiently generic $x, y$. Therefore, we combine (\ref{equation: quasirandom to edge uniform, 3}) and (\ref{equation: quasirandom to edge uniform, 5}) to obtain
\begin{IEEEeqnarray*}{rCl}
\Delta^{2^n} & \leq & O(|\Om(V)|^{2^{n} - 2})\sum_{(x, y) \in \Om(*)^2}  \sum_{D(\Lambda_{xy})} \prod_{\iota \in 2^{V'}} g^\iota \\
& = & O(|\Om(V)|^{2^{n} - 2}) \sum_{D(\Om)}\prod_{\iota \in 2^V} f^\iota . 
\end{IEEEeqnarray*}
This completes the induction. Now the lemma follows   from $\e$-quasirandomness. 
\end{proof}

We deduce the following hypergraph regularity lemma in the \'etale setting:

\begin{theorem}\label{theorem: combinatorial hypergraph regularity with covers}
Let $\Om$ be an irreducible   system of difference varieties on $V$ over $A$.  For every $u \subseteq V$ with $|u| \geq 2$, let $\e_u(q) = q^{-2^{-|u|-1}}$ and let $\delta_u(q) = q^{-1/2}$. Let $\e = (\e_u)_{|u| \geq 2}$ and $\delta = (\delta_u)_{|u| \geq 2}$.

Then there is a surjective refinement $\Om'$ of $\Om$ over a finitely generated   extension $A'$ of $A$ with irreducible components $\Om_1,\ldots, \Om_n$    such that each $\Om_i$ is $\delta$-quasirandom   and $\e$-edge-uniform. $n$ depends only on the degrees of the varieties in $\Om$.
\end{theorem}

\noindent \textbf{Remark:} Theorem \ref{theorem: combinatorial hypergraph regularity with covers} gives weaker bounds for edge-uniformity in the case $n = 2$ than the original result of \cite{Tao2012}, which proves $q^{-1/4}$-edge-uniformity, rather than $q^{-1/8}$-edge-uniformity. Tao's proof of the algebraic regularity lemma does not seem to generalise easily to the higher dimensional setting, even with the stationarity theorem. This is why we have taken the route of quasirandomness, which involves a loss in the error bounds. It would be interesting to know if  our bound can be improved.

  \medskip
  
\subsection{A classical algebraic hypergraph regularity lemma}  \label{subsection: classical hypergraph regularity}
  
\emph{We fix $A$ a perfect inversive finitely generated difference field. $V$ is a finite set.}
\medskip  
  
In this section, we use chain decompositions of systems of varieties to deduce an algebraic hypergraph regularity lemma which does not use the \'etale point of view. This algebraic hypergraph regularity lemma retains the main combinatorial features of Theorem \ref{theorem: combinatorial hypergraph regularity with covers} while eliminating references to refinements, but the  trade-off is that we lose definability of the hypergraph partitions. 
For this reason,  the definitions that we have already set up do not apply in this setting exactly, so  we refrain from using any prior notions except  chain decompositions.

We   prove two technical lemmas. The first lemma revisits Proposition \ref{proposition: regular refinement of system of varieties}. In that proposition, we constructed a surjective refinement of $\Om$ with regular components. However, our definition of surjective refinements only requires surjectivity at the top level $V$. Here we need to construct a chain decomposition of $\Om$, so we need a slightly different notion.

\begin{lemma}\label{lemma: better surjective refinement}
Let $\Om$ be an irreducible system of varieties on $V$ over $A$. There is a system of varieties $\Om_0$ on $V$ over $A$ and a surjective   $P(V)^-$-refinement $\Om_1$ of $\Om_0$ over a finitely generated extension $A'$ of $A$ satisfying the following:
\begin{enumerate}
\item for every $u \subseteq V$, $\Om(u) \subseteq \Om_0(u)$
\item for every $u \subseteq V$, the projections $\pi : \Om_1(u) \to \Om_0(u)$,  $ \Om_1(u) \to \Om_1(u)^-$ and $\Om_0(u) \to \Om_0(u)^-$ are  generically surjective. 
\item for every $u \subseteq V$ and every section $S$ of the projection $\pi : \Om_1(u)^- \to \Om_0(u)^-$, the projection $\rho_u^{-1}(S) \to \Om_0(u)$ is generically surjective\label{equation: better surjective refinement: surjectivity over sections}
\item every irreducible component of $\Om_1$ is regular
\end{enumerate}
\end{lemma}

\begin{proof} 
Let $\S$ be a system of difference fields associated to $\Om$. $\S$ is unique up to isomorphism since $\Om$ is irreducible. For every $u \subseteq V$, let $\mathcal{P}_u$ be a family of polynomials such that $\S(u)$ is the splitting field of $\mathcal{P}_u$ over $\S(u)^-$. Adding purely inseparable polynomials to $\mathcal{P}_u$ if necessary, we can assume that $\mathcal{P}_u$ is a family of polynomials over $A(x_u)_\si$, where $x_u$ is the coordinate of a point of $\Om(u)$.

For $|u| = 1$, set $\Om_0(u) = \Om(u)$. For general $u$, define $\Om_0(u)$ to be the algebraic cover of $\Om_0(u)^-$ consisting of points $(a, b)$ where $a \in \Om_0(u)^-$ and $b$ generates the splitting field of $\mathcal{P}_u$ over $A(a)_\si$. $\Om_0(u)$ adds only algebraic information to $\Om_0(u)^-$ and does not choose between difference field structures on its generic points. We can assume that $\Om(u)$ is a subvariety  of $\Om_0(u)$ by using some appropriate syntactical coding. It is clear that the projections $\rho_u : \Om_0(u) \to \Om_0(u)^-$ and $\Om_0(u) \to \Om_0(v)$ are all surjective.

By Proposition \ref{proposition: existence of regular refinements}, find a regular $P(V)^-$-refinement $\S'$ of $\S$. For every $u\subseteq V$, let $\mathcal{Q}_u$ be a family of polynomials such that $\S'(u)$ is the splitting field of $\mathcal{Q}_u$ over $\S(u)$. Similarly to the construction sketched above, define $\Om_1(u)$ to be a variety projecting onto $\Om_0(u)$ such that the generic points of $\Om_1(u)$ generate the roots of $\mathcal{Q}_u$ over the generic points of $\Om_0(u)$ but $\Om_1(u)$ does not impose any difference structure on these roots.  As in Proposition \ref{proposition: regular refinement of system of varieties}, since the field extensions all lie in the invariant algebraic closure, regularity does not depend on the difference field structure and every component of $\Om_1$ is regular.

The surjective properties between $\Om_1$ and $\Om_0$ follow by construction.
\end{proof}

Our second lemma constructs a chain decomposition of the refinement obtained in Lemma \ref{lemma: better surjective refinement} with useful combinatorial properties. In the following, with $\Om_0$ and $\Om_1$ as in Lemma \ref{lemma: better surjective refinement}, we assume that the projections $\Om_1(u) \to \Om_0(u)$, $\Om_1(u) \to \Om_1(u)^-$  and $\Om_0(u) \to \Om_0(u)^-$ are exactly surjective, as this amounts to  changes in the varieties $\Om_1(u), \Om_0(u)$ of size $O(q^{dim(\Om(u)) - 1})$. These changes will eventually be absorbed in the error terms.

\begin{lemma}\label{lemma: chain decomposition for regularity theorem}
Let $\Om$ be an irreducible system of varieties on $V$ over $A$ and take $\Om\subseteq \Om_0$ and $\Om_1$ a surjective refinement of $\Om_0$ as in Lemma \ref{lemma: better surjective refinement}. Specialise $\Om, \Om_0, \Om_1$ to some $K_q$.

 There is a $P(V)$-chain decomposition $(\mathcal{W}_v)_{v \in P(V)}$ of $\Om$ with size depending only on the degrees of the varieties in $\Om$ such that for every $v \subseteq V$ and every $X\in \mathcal{W}_v$, there is a set $X^* \subseteq \Om_1(v)$ such that the following properties are satisfied:
\begin{enumerate}
\item For every $u \subseteq V$ and $X \in \mathcal{W}_u$, $|X| \geq \lambda_u |\Om(u)|$ for some scalar $\lambda_u> 0$ which does not depend on $q$
\item  For every $u \subseteq V$, the sets $X^*$ for $X \in \mathcal{W}_u$ are all contained in different irreducible components of $\Om_1(u)$ and for every $X \in \mathcal{W}_u$, the projection $\pi: X^* \to X$ is bijective
\item For every $u \subseteq V$, if $W$ is a $P(u)$-chain in $(\mathcal{W}_v)_{v \in P(u)}$, the sets $W(v)^*$ form a $P(u)$-chain in $\Om_1\upharpoonright P(u)$ (and we write $W^*(v)$ instead of $W(v)^*$)
\item For every $u \subseteq V$, if $W$ is a $P(u)$-chain in $(\mathcal{W}_v)_{v \in P(u)}$, writing $Z(u)$ for the irreducible component  of $\Om_1(u)$ containing $W^*(u)$, we have      $\rho_uW^*(u) = \rho_uZ(u) \cap W^*(u)^-$
\item  For every $u \subseteq V$ and $X, X' \in \mathcal{W}_u$, the sets  $\rho_uX$ and $\rho_uX'$ are either equal or disjoint.\label{equation: chain decomposition: further compatibility with projections}
\end{enumerate}
\end{lemma}

\begin{proof}
We construct the families $(\mathcal{W}_v)_{ P(V)^-}$ and $(\mathcal{W}_v^*)_{ P(V)^-}$ inductively. For $|v| = 1$, we choose components $Z_1(v), \ldots, Z_p(v)$ of $\Om_1(v)$ such that the sets $\pi(Z_i(v))$ partition $\Om(v)$. Let $\mathcal{W}_v = \{\pi(Z_i(v))\}$ and define $\mathcal{W}_v^*$ to be a set of sections of the projections $\pi : Z_i(v) \to \Om(v)$.

 Suppose $(\mathcal{W}_v)_{I}$ and $(\mathcal{W}_v^*)_{ I}$ have been constructed where $I$ is downward closed and take $u \subseteq V$ such that $P(u)^- \subseteq I$. Let $W_1, \ldots, W_n$ be an enumeration of the $P(u)^-$-chains contained in $(\mathcal{W}_v)_{P(u)^-}$ which intersect $\rho_u \Om(u)$. Note that the sets $W_i^*(u)$ all lie in different irreducible components of $\Om_1(u)^-$ by Lemma \ref{lemma: measure of system of varieties is measure of product}.
 
We consider  any $P(u)^-$-chain $W$ among the chains $W_1, \ldots, W_n$. Then $\pi : W^*(u)  \to W (u) $ is bijective and by property (\ref{equation: better surjective refinement: surjectivity over sections}) of  Lemma \ref{lemma: better surjective refinement}, $\pi$ lifts to a surjection  $\pi: \rho_u^{-1}(W^*(u)) \to \rho_u^{-1}(W(u))$. Let $Z_1(u), \ldots, Z_{k}(u)$ be irreducible subvarieties of $\Om_1(u)$ such that $\rho_u^{-1}(W^*(u)) \cap \pi^{-1}(\Om(u)) \subseteq \bigcup Z_i(u)$ and each $Z_i(u)$  projects to a subset of $\Om(u)$ under $\pi$.

For every $i$, since $Z_i(u)$ is irreducible, both projections $\rho_u : Z_i(u) \to \Om_1(u)^-$ and $\pi: Z_i(u) \to \Om(u)$ have constant multiplicity. Moreover, for every fibre $F$ of $\rho_u : Z_i(u) \to \Om_1(u)^-$, the image $\pi(F)$ has constant size. 

Starting with $i = 1$, we choose sections $X^*_1, \ldots, X^*_m$ of the projection 
 \[
Z_1(u) \cap \rho_u^{-1}(W^*(u)) \to W^*(u)
\] 
such that $\pi( \bigcup X_j^*) = \pi \big(Z_1(u) \cap \rho_u^{-1}(W^*(u)) \big)$ and the sets $\pi( X_j^*)$ are pairwise disjoint. We add the sets $X_j^*$ to $\mathcal{W}_u^*$ and the sets $\pi(X_j^*)$ to $\mathcal{W}_u$. Property  (\ref{equation: chain decomposition: further compatibility with projections}) can be obtained by the same argument.

For $i =2$, suppose that some element of $Z_2(u) \cap \rho_u^{-1}W^*(u)$ projects into the set $\pi(Z_1 \cap \rho_u^{-1} W^*(u))$. By irreducibility of $Z_2(u)$, it follows that 
\[
\pi(Z_2(u) \cap \rho_u^{-1}W^*(u)) \subseteq \pi(Z_1(u) \cap \rho_u^{-1}W^*(u))
\]
so that we can ignore $Z_2(u)$ and move on to $Z_3(u)$. If this does not happen, then we can choose sections of $Z_2(u)$ as for $Z_1(u)$. Property (\ref{equation: chain decomposition: further compatibility with projections})  is guaranteed by the same argument.

We iterate in this way through all of $Z_1(u), \ldots, Z_k(u)$, thus constructing a partition $\mathcal{W}_u$ of $\rho_u^{-1}(W(u))$ with the desired properties. We then iterate this construction through all of $W_1(u), \ldots, W_n(u)$ and this defines $\mathcal{W}_u$ and $\mathcal{W}_u^*$ as desired.
\end{proof}

We can now prove an algebraic hypergraph regularity lemma which retains the   strong combinatorial properties of the \'etale setting. The theorem is stated for  arbitrary systems of   varieties, but the main case of interest is for systems $\Om$ on $V$ where $\Om(V)^-$ is the Cartesian product $\prod_{i \in V} \Om(i)$ and $\Om(V)$ is a cover of  $\Om(V)^-$ giving rise to an arbitrary definable set.

We state the theorem in a slightly redundant way: condition (2) directly implies condition (1), as the proof will show. Nevertheless, condition (1) is probably the main import  of the theorem. 

\begin{theorem}\label{theorem: combinatorial regularity, classical partitions}
Let  $\Om$ an irreducible   system of difference varieties on $V$ over $A$. Specialise $\Om$ to some $K_q$ and let $(\mathcal{W}_u)_{u \in P(V)}$ be  a  $P(V)$-chain decomposition constructed as  in Lemma \ref{lemma: chain decomposition for regularity theorem}.  Then the following properties hold: 
\begin{enumerate}
\item For every $P(V)^-$-chain $W$ contained in $(\mathcal{W}_v)_{v \in P(V)^-}$, either $W(V) \cap \rho_V \Om(V)$ is small, i.e.
\[
|W(V) \cap \rho_V \Om(V)| = O(q^{|V|-1})
\]
 or,  for every chain $W'$ contained in $W$, \[
|\frac{|\rho_V \Om(V) \cap W'(V)|}{|\rho_V\Om(V) \cap W(V) |} - \frac{|W'(V)|}{|W(V)|}| = O(q^{-2^{-|V|-1}})
\]\label{equation: combinatorial regularity, 1}
\item for every $u \subseteq V$, for every $P(u)$-chain $W$ contained in $(\mathcal{W}_v)_{v \in P(u)}$,  and for every  $P(u)^-$-chain $W'$ contained in $W \upharpoonright P(u)^-$, 
\[
|\frac{|\rho_u W(u)\cap W'(u)|}{|\rho_u W(u)  |} - \frac{|W'(u)|}{|W(u)^-|}| = O(q^{-2^{-|u|-1}})
\]\label{equation: combinatorial regularity, 2}
  \item For every $u \subseteq V$, for every $P(u)$-chain $W$ contained in $(\mathcal{W}_v)_{v \in P(u)}$,  writing $f = \I_{\rho_u W(u) } - \frac{|\rho_u W(u)|}{|W(u)^-|}\I_{W(u)^-}$ on $W(u)^-$, 
\[
\sum_{D(\Om \upharpoonright P(u)^-)} \prod_{\iota \in 2^u}f^\iota  = O(q^{-2^{-|u|-1}})|\Om(u)|^2
\]\label{equation: combinatorial regularity, 3}
\end{enumerate}
where the functions $O(\cdot)$ depend only on the degrees of the   polynomials in  $\Om$.

%
%
%
% every chain $W$ in $(\mathcal{W}_u)$  we have:
% 
%\begin{enumerate}
%\item for every chain $W'$ contained in $W$, \[
%|\frac{|\phi \cap W'(V)|}{|\phi \cap W(V) |} - \frac{|W'(V)|}{|W(V)|}| = O(q^{-2^{-|V|-1}})
%\]\label{equation: combinatorial regularity, 1}
%\item for all $u \in P(V)^-$ and   for every $P(u)^-$-chain $W'$ contained in $W\upharpoonright P(u)^-$, 
%\[
%|\frac{|\rho_u\Om(u) \cap W'(u)|}{|\rho_u\Om(u) \cap W(u)^-|} - \frac{|W'(u)|}{|W(u)^-|}|  = O(q^{-2^{-|u|-1}})
%\]\label{equation: combinatorial regularity, 2}
%\item Writing $f = \I_{\phi \cap W(V)} - \frac{|\phi \cap W(V)|}{|W(V)|}\I_{W(V)}$ on $W(V)$, 
%\[
%\sum_{D(\Om)} \prod_{\iota \in 2^V}f^\iota  = O(q^{-2^{-|V|-1}})|\Om(V)|^2
%\]\label{equation: combinatorial regularity, 3}
%\item For every $u \in P(V)^-$, writing $f_u = \I_{\rho_u W(u)} - \frac{|\rho_u W(u) |}{|W(u)^-|}\I_{W(u)^-}$ on $W(u)^-$, 
%\[
%\sum_{D(\Om(u)\upharpoonright P(u)^-)} \prod_{\iota \in 2^u}f_u^\iota  = O(q^{-2^{-|u|-1}})|\Om(u)^-|^2
%\]\label{equation: combinatorial regularity, 4}
%\end{enumerate}
%where $O(\cdot)$ depends only on the degrees of the difference polynomials in $\phi$ and $\Om$.
\end{theorem}

\begin{proof}
Let $\Om_0$ and $\Om_1$ be as in Lemma \ref{lemma: better surjective refinement}. For every $X \in \mathcal{W}_v$, we write $X^*$ and $\pi : X^* \to X$ as in Lemma \ref{lemma: chain decomposition for regularity theorem}.

 Take $u \subseteq V$ and $W$ a $P(u)$-chain in $(\mathcal{W}_v)_{P(u)}$. For every $v \subseteq u$, write $Z(v)$ for the irreducible subvariety of $\Om_1(v)$ containing $W^*(v)$. Then $Z$ is a regular system of varieties.   Since $\rho_u W^*(u) = \rho_u Z(u) \cap W^*(u)^-$ and by Theorem \ref{theorem: combinatorial hypergraph regularity with covers}, $\rho_u W^*(u)$ has the obvious edge-uniformity properties with respect to chains contained in $W^* \upharpoonright P(u)^-$. 
 
  Let $W'$ be a $P(u)^-$-chain contained in $W \upharpoonright P(u)^-$. Write $W''(v)$ for the preimage of $W'(v)$ in $W^*(v)$, for every $v \in P(u)^-$, so that $W'(u) = \pi(W''(u))$. Then we have 
  \[
  |\frac{|\rho_u W^*(u) \cap W''(u)|}{|\rho_u W^*(u)|} - \frac{|W''(u)|}{|W^*(u)^-|}| = O(q^{-2^{-|u|-1}})
  \]
Moreover, since $\pi$ restricts to bijections $W^*(u)^- \to W(u)^-$ and $W^*(u) \to W(u)$, it follows easily that
  \[
  |\frac{|\rho_u W(u) \cap W'(u)|}{|\rho_u W(u)|} - \frac{|W'(u)|}{|W(u)^-|}| = O(q^{-2^{-|u|-1}})
  \] 
giving property (\ref{equation: combinatorial regularity, 2}). Property (\ref{equation: combinatorial regularity, 3}) follows by a similar argument and an application of Theorem \ref{theorem: combinatorial hypergraph regularity with covers}.

Finally, by Property (\ref{equation: chain decomposition: further compatibility with projections}) of Lemma \ref{lemma: chain decomposition for regularity theorem}, $\rho_V(\Om(V))$ is the disjoint union of a family of sets $\rho_V W_1(V), \ldots, \rho_V W_n(V) $ where the $W_i(V)$ are elements of $ \mathcal{W}_V$. Applying property (\ref{equation: combinatorial regularity, 2}) to each of the sets $\rho_V W_i(V)$ gives property (\ref{equation: combinatorial regularity, 1}). 
 \end{proof}  
   It is an open question whether we can strengthen  Theorems \ref{theorem: combinatorial hypergraph regularity with covers} and \ref{theorem: combinatorial regularity, classical partitions} to obtain an algebraic hypergraph regularity lemma such that the chain decomposition of Theorem \ref{theorem: combinatorial regularity, classical partitions} comes from definable sets.

 The results of \cite{BeyarslanHrushovski2012} about the model theory of pseudofinite fields show that an algebraic hypergraph regularity of this kind is available in many pseudofinite fields.
 \cite{BeyarslanHrushovski2012} shows that for almost all completions $T'$ of the theory $T$ of pseudofinite fields, for $K \models T'$, $M \prec K$ and $M \subseteq A \subseteq K$, we have $\acl(A) = \dcl(A)$. Here,  ``almost all'' is meant with respect to the Haar measure on the absolute Galois group of $\Qq$ or $\mathbb{F}_p$.  In this setting, the sections of Lemma \ref{lemma: chain decomposition for regularity theorem} become uniformly definable sets and hence the chain decomposition of  \ref{theorem: combinatorial regularity, classical partitions} is definable in the classical sense.
 
  An algebraic  hypergraph regularity lemma in the setting of \cite{BeyarslanHrushovski2012} can be found in the PhD thesis of the second named author. The proof of the algebraic regularity lemma given in there is rather different from the proof given here. The  argument used a weak version of the stationarity theorem and the derivation of the algebraic regularity lemma relied more heavily on combinatorial arguments. The argument in that thesis  also gave weaker error bounds on regularity than the ones we obtain here.

\bibliographystyle{alpha}\bibliography{Pseudofiniteregularity}

\begin{thebibliography}{CvdDM92}

\bibitem[BH12]{BeyarslanHrushovski2012}
{\"O}zlem Beyarslan and Ehud Hrushovski.
\newblock On algebraic closure in pseudofinite fields.
\newblock {\em The Journal of Symbolic Logic}, 77(4):1057--1066, 2012.

\bibitem[CH99]{ChatzidakisHrushovski1999}
Zo{\'e} Chatzidakis and Ehud Hrushovski.
\newblock Model theory of difference fields.
\newblock {\em Transactions of the American Mathematical Society},
  351(8):2997--3071, 1999.

\bibitem[CH08]{ChatzidakisHrushovski2008}
Zo{\'e} Chatzidakis and Ehud Hrushovski.
\newblock Difference fields and descent in algebraic dynamics. i.
\newblock {\em Journal of the Institute of Mathematics of Jussieu},
  7(4):653--686, 2008.

\bibitem[Coh65]{Cohn1965}
Richard~M. Cohn.
\newblock {\em Difference algebra}.
\newblock Interscience tracts in pure and applied mathematics ; no. 17.
  Interscience Publishers, New York, 1965.

\bibitem[CS16]{ChernikovStarchenko2016}
Artem Chernikov and Sergei Starchenko.
\newblock Definable regularity lemmas for nip hypergraphs.
\newblock {\em The Quarterly Journal of Mathematics}, 07 2016.

\bibitem[CT20]{ChernikovTowsner2020}
Artem Chernikov and Henry Towsner.
\newblock Hypergraph regularity and higher arity vc-dimension, 2020.

\bibitem[CvdDM92]{ChatzidakisVanDDMacintyre1992}
Zoe Chatzidakis, Lou van~den Dries, and Angus Macintyre.
\newblock Definable sets over finite fields.
\newblock {\em Journal fur die Reine und Angewandte Mathematik},
  1992(427):107--136, 1992.

\bibitem[DT17]{DzamonjaTomasic2017}
Mirna D{\v z}amonja and Ivan Toma{\v s}i{\'c}.
\newblock Graphons arising from graphs definable over finite fields, 2017.

\bibitem[Gow06]{Gowers2006}
W.~T. Gowers.
\newblock Quasirandomness, counting and regularity for 3-uniform hypergraphs.
\newblock {\em Combinatorics, Probability and Computing}, 15(1-2):143, jan
  2006.

\bibitem[Gow07]{Gowers2007}
W.~T Gowers.
\newblock Hypergraph regularity and the multidimensional szemer{\'e}di theorem.
\newblock {\em Annals of mathematics}, 166(3):897--946, 2007.

\bibitem[Hru02]{Hrushovski2002}
Ehud Hrushovski.
\newblock Pseudo-finite fields and related structures.
\newblock In {\em Model theory and applications}, volume~11 of {\em Quad.
  Mat.}, pages 151--212. Aracne, Rome, 2002.

\bibitem[Hru06]{Hrushovski2006}
Ehud Hrushovski.
\newblock Groupoids, imaginaries and internal covers.
\newblock {\em Turkish Journal of Mathematics}, 36, 03 2006.

\bibitem[Hru22]{HrushovskiFrob}
Ehud Hrushovski.
\newblock The elementary theory of the frobenius automorphisms, 2022.

\bibitem[Joh19]{Johnson2019}
Will Johnson.
\newblock Counting mod n in pseudofinite fields, 2019.

\bibitem[Lev08]{Levin2008}
Alexander Levin.
\newblock {\em Difference Algebra}, volume~8 of {\em Algebra and Applications}.
\newblock Springer, 2008.

\bibitem[PS13]{Pillay2013}
Anand Pillay and Sergei Starchenko.
\newblock Remarks on tao's algebraic regularity lemma, 2013.

\bibitem[RS04]{RodelSkokan2004}
Vojt{\v e}ch R{\"o}dl and Jozef Skokan.
\newblock Regularity lemma for k-uniform hypergraphs.
\newblock {\em Random Structures \& Algorithms}, 25(1):1--42, 2004.

\bibitem[RT06]{RytenTomasic2006}
Mark Ryten and Ivan Toma{\v s}i{\'c}.
\newblock Acfa and measurability.
\newblock {\em Selecta Mathematica}, 11:523--537, 04 2006.

\bibitem[Sze78]{Szemeredi1976}
Endre Szemer\'edi.
\newblock Regular partitions of graphs.
\newblock In {\em Colloques Internationaux du CNRS}, volume Probl{\`e}mes
  combinatoires et th{\'e}orie des graphes, Orsay, 1976. {\'E}ditions du Centre
  national de la recherche scientifique, 1978.

\bibitem[Tao12]{Tao2012}
Terence Tao.
\newblock Expanding polynomials over finite fields of large characteristic, and
  a regularity lemma for definable sets.
\newblock {\em Contributions to Discrete Mathematics}, 10, 11 2012.

\bibitem[Tom06]{Tomasic2006}
Ivan Toma{\v s}i{\'c}.
\newblock Independence, measure and pseudofinite fields.
\newblock {\em Selecta Mathematica}, 12(2):271, 2006.

\bibitem[TW21]{TerryWolf2021}
C.~Terry and J.~Wolf.
\newblock Higher-order generalizations of stability and arithmetic regularity,
  2021.

\end{thebibliography}

\end{document}